\documentclass[11pt]{amsart}
\usepackage{}
\usepackage{xcolor}
\usepackage{amssymb}
\usepackage{amsmath}
\usepackage{times}

\newcommand{\s}{\sigma}
\newcommand{\tl}{\tilde{\lambda}}
\newcommand{\la}{\lambda}
\newcommand{\de}{\delta}

\newtheorem{theorem}{Theorem}[section]
\newtheorem{lemma}[theorem]{Lemma}

\newtheorem{corollary}[theorem]{Corollary}

 \theoremstyle{definition}

\theoremstyle{remark}
\newtheorem{remark}[theorem]{Remark}

\numberwithin{equation}{section}

\begin{document}

\title[sums of Hessian operators]
{Pogorelov interior estimates for sum-of-Hessians equations}


\begin{abstract}
	In this paper, we develop a new approach to sum-of-Hessians equations involving multiple $k$-Hessian operators of distinct orders. By exploiting the concavity of sums of Hessian operators, we derive Pogorelov estimates for the corresponding equations under the dynamic semi-convexity condition. Moreover, when the highest order is $n-1$ or $n$, we establish such estimates for admissible solutions without relying on this condition, and they are therefore optimal. As an application, when the right-hand side is identically $1$, we prove that every entire admissible solution in $\mathbb{R}^n$ with quadratic growth is necessarily a quadratic polynomial.



\end{abstract}

\author{Weisong Dong}
\address{School of Mathematics, Tianjin University,
	Tianjin, 300354, China}
\email{dr.dong@tju.edu.cn}

\author{Sirui Xu}
\address{School of Mathematics, Tianjin University,
	Tianjin, 300354, China}
\email{srxu@tju.edu.cn}

\author{Ruijia Zhang}
\address{Department of Mathematics, Sun Yat-sen University, Guangzhou, 510275, China}
\email{zhangrj76@mail.sysu.edu.cn}
\maketitle

\section{Introduction}


A milestone in the regularity theory of Monge-Amp\`ere equations is the celebrated interior $C^2$ estimate established by Pogorelov \cite{Po1} in 1971 under homogeneous boundary conditions. This estimate was a crucial ingredient in his proof of the Liouville-type theorem for the Monge-Amp\`ere equation in \cite{Po2}. Now referred to as the \emph{Pogorelov-type interior $C^2$ estimate} (see also \cite{Yuan}), its necessity is highlighted by Pogorelov's counterexample \cite{Po3}, which shows that pure interior $C^2$ estimates generally fail when $n \geqslant 3$. In two dimensions, however, these estimates remain valid, as first proven by Heinz \cite{Heinz} (cf. \cite{CHO, Liu}).

Pogorelov-type interior $C^2$ estimates have also been extended to various classes of fully nonlinear equations. Notable generalizations include the $k$-Hessian equations \cite{CW,SUW}, the $p$-Monge-Ampère equations \cite{CJ, Dinew, Dong}, the Hessian quotient equations \cite{LT2}, and the dual $k$-Hessian equations \cite{CJ, CDH}. These estimates play a pivotal role not only in the regularity theory of these operators but also in establishing associated Liouville-type theorems.

As in \cite{LRW}, a Pogorelov-type interior $C^2$ estimate that is independent of global gradient bounds is essential for proving Liouville-type theorems for the $k$-Hessian equation$$\sigma_k (D^2 u) = 1 \quad \text{in } \mathbb{R}^n,$$where $\sigma_k$ denotes the $k$-th elementary symmetric function. In what follows, we briefly review the development of Liouville-type theorems for $k$-Hessian equations.

A classical result due to Jörgens \cite{Jo}, Calabi \cite{Ca}, and Pogorelov \cite{Po2} states that any smooth, strictly convex solution to the Monge-Amp\`ere equation $\det(D^2u) = 1$ in $\mathbb{R}^n$ must be a quadratic polynomial; see also \cite{CYau, TruW}. Since then, substantial effort has been devoted to extending this Liouville-type property to $k$-Hessian equations. A pioneering step in this direction was taken by Bao, Chen, Guan, and Ji \cite{BCGJ}, who established the result under the additional assumptions that the solutions are strictly convex and satisfy a quadratic growth condition.

It is natural to expect the Liouville property for $k$-Hessian equations to hold under weaker assumptions--namely, that the solution is either semi-convex (as conjectured in \cite{CY}) or merely $k$-admissible and satisfies a quadratic growth condition (as suggested in \cite{LRW}, in view of Warren's counterexample \cite{Warren}). Regarding the first direction, under an almost convexity assumption, the result was previously established by Chang and Yuan \cite{CY} for the $2$-Hessian equation $\sigma_2 (D^2u) = 1$ in $\mathbb{R}^n$. Recently, Shankar and Yuan \cite{SY} successfully removed this restriction, proving that any smooth, semi-convex solution must be a quadratic polynomial.
In the second direction, Li, Ren, and Wang \cite{LRW} established the Liouville-type theorem for the $k$-Hessian equation $\sigma_k (D^2u) = 1$ in $\mathbb{R}^n$ under the additional assumption that the solution is $(k+1)$-convex. This result was subsequently generalized by Chu and Dinew \cite{CD} and Zhang \cite{Zhang}. More precisely, Chu and Dinew assumed the relaxed condition $\sigma_{k+1}(D^2u) \geqslant -K$ (which corresponds to $K=0$ in \cite{LRW}), whereas Zhang instead incorporated the semi-convexity assumption $D^2u \geqslant -KI$.

We note that the methodologies employed in these two directions are fundamentally distinct. In the first direction, Shankar and Yuan established a Jacobi inequality for the shifted trace and subsequently utilized the Legendre-Lewy transformation to convert the $2$-Hessian equation into a uniformly elliptic one. In the second direction, a crucial step involves establishing a Pogorelov-type interior $C^2$ estimate that is independent of global gradient bounds, as demonstrated in \cite{LRW}. To derive a key concavity inequality for the $k$-Hessian operator, the authors in \cite{LRW} required the solution to be $(k+1)$-convex. Zhang \cite{Zhang} subsequently relaxed this $(k+1)$-convexity requirement to semi-convexity while still preserving the necessary concavity inequality. By contrast, Chu and Dinew \cite{CD} bypassed this path by deriving the Pogorelov-type interior $C^2$ estimate via a Pogorelov-type interior gradient estimate under the lower bound assumption $\sigma_{k+1}(D^2u) \geqslant -K$.

The purpose of this paper is to generalize the existing Pogorelov-type interior $C^2$ estimates to a wider class of fully nonlinear elliptic Hessian equations--namely, equations involving linear combinations of multiple $k$-Hessian operators under substantially weaker conditions.

Now, we introduce some notation.
Let $m, k,n \in \mathbb{N}$, $m<k \leqslant n$ and $n\geqslant 3$, $a=(a_1, \dots, a_m) \in \mathbb{R}^m$, and $\lambda \in \mathbb{R}^n$. Define
\begin{align}\label{sumH}
    F (\la) =\s_k^{(n)} (\la)+\sum_{r=1}^m a_r \s_{k-r}^{(n)}(\la),
\end{align}
where $\sigma_k^{(n)}$ denotes the $k$-th elementary symmetric function on $\mathbb{R}^n$,
with the convention that $\sigma_0^{(n)} = 1$ and $\sigma_i^{(n)} = 0$ for $i < 0$ or $i > n$.
Let $\la (D^2 u) = (\la_1, \dots, \la_n)$ denote the eigenvalues of the Hessian $D^2 u$ of a smooth function
$u : \Omega \rightarrow \mathbb{R}$, where $\Omega \subset \mathbb{R}^n$ is a domain.
For convenience, we define
\[F(D^2 u) := F(\la(D^2 u) ). \]

We consider the following equations involving sums of Hessian operators with vanishing Dirichlet boundary value in general form,
\begin{equation}\label{eq1}
\left\{\begin{aligned}
F(D^2 u) &= \psi(x,u, \nabla u) \quad &\text{in} \ &\Omega, \\
u&=0 \quad&\text{on} \ &\partial \Omega.
\end{aligned}\right.
\end{equation}
Here $\psi\in C^2(\overline{\Omega}\times \mathbb{R}\times \mathbb{R}^n)$ satisfies $\psi>\psi_0>0$ for a constant $\psi_0>0$ in $\overline{\Omega}$.
Besides, we investigate smooth entire solutions in $\mathbb{R}^n$ satisfying equation:
\begin{equation}
	\label{sigma-k}
	F(\lambda(D^2u)) = 1, \quad \text{in } \mathbb{R}^n,
\end{equation}
where $\lambda(D^2 u) = (\lambda_1, \ldots, \lambda_n)$ denotes the eigenvalues of the Hessian matrix $D^2 u$ of a smooth function $u$ on $\mathbb{R}^n$.

The sum-of-Hessians equation \eqref{eq1} can be viewed as a natural generalization of the $k$-Hessian equation \eqref{khessian}, which is deeply connected to numerous problems in differential geometry, convex geometry and complex geometry. The study of second-order a priori estimates for equation \eqref{khessian} has been a longstanding problem that remains open for $3\leqslant k \leqslant n-3$ (cf. \cite{GRW, RW1, RW2}). Recently, the sum-of-Hessians equation has attracted increasing attention. Harvey and Lawson \cite{HL} investigated special Lagrangian equations, which fall into this category and can be reduced to a class of sum-of-Hessians equations. Similar types of equations were also studied by Krylov \cite{Kr} and Dong \cite{D}, who established $C^{1,1}$ estimates by exploiting the concavity of the corresponding operators. More recently, Guan and Zhang \cite{GZ} considered curvature estimates for these equations on hypersurfaces. Beyond these applications, equations of this type appear in the study of the Fu-Yau equation and Mirror Symmetry within complex geometry (see, e.g., the references in \cite{GZ}).

Equations \eqref{eq1} and \eqref{sigma-k} are elliptic for solutions in an admissible set.
Recall that the admissible set for the $k$-Hessian equation, the G{\aa}rding cone $\Gamma_k^{n} \subset \mathbb{R}^n$ is an open symmetric convex cone with a vertex at the origin, defined by
	\begin{align}\label{gk}
	\Gamma_k^{n}=\lbrace \lambda\in\mathbb{R}^{n} \;|\; \s_i^{(n)} (\lambda)>0, \;\forall\; i= 1,\ldots,k\rbrace.
	\end{align}
The key observation here is that the operator \eqref{sumH} can be interpreted as a $k$-Hessian operator in a higher-dimensional space $\mathbb{R}^{n+m}$, provided that the following Real Root Hypothesis holds.
\begin{enumerate}
    \item[(RR)] \label{RR}
    The polynomial of degree $m$
    \[
    P(t)=t^m + \sum_{r=1}^m (-1)^r a_r t^{m-r}
    \]
    has $m$ real roots $y_i \in \mathbb{R}$, $i=1,\dots,m$.
\end{enumerate}

Hypothesis (RR) was proposed by Li-Ren-Wang \cite{LRW2} to coincide with the Quotient Concavity condition. In that work, they investigated curvature estimates for closed convex hypersurfaces satisfying a sum-of-Hessians curvature equation. We provide an alternative interpretation of this hypothesis by placing sum-of-Hessian operators within a higher-dimensional space. 

Denote $y = (y_1, \dots, y_m)$.
We study the problem within the admissible set of \eqref{eq1}, defined by
\[
\Gamma_k^{n+(m)}=\lbrace\lambda \in \mathbb{R}^{n} \;|\;(\la, y) \in \Gamma_k^{n+m}\rbrace.
\]
Clearly, $\Gamma_k^{n+(m)}$ is a convex set, and $F$ is elliptic in this set since 
$\Gamma_k^{n+(m)} \times \{y\} \subset \Gamma_k^{n+m}$. 
Moreover, in the set $\Gamma_k^{n+(m)}$,
\begin{align}\label{sl}
    \sigma_l^{(n+m)}(\lambda, y)
    = \sigma_l^{(n)}(\lambda)
      + \sum_{r=1}^m a_r \sigma_{l-r}^{(n)}(\lambda)
    > 0, 
    \quad l = 1, \dots, k.
\end{align}

To derive the estimates, we require the solution to satisfy one of the following Conditions: 
\begin{enumerate}
    \item \label{con1}
    $\lambda \in \Gamma_k^{n+(m)}$ and $y_i \geqslant  0$ for all $i = 1, \dots, m$;
    
    \item \label{con2}
    $\lambda \in \Gamma_{k-1}^n$ and $a_i \geqslant 0$ for all $i = 1, \dots, m$.
\end{enumerate}

When $m=0$, \eqref{sumH} reduces to the $k$-Hessian equation
\begin{align}\label{khessian}
\sigma_k(\lambda) = \psi(x,u,\nabla u).
\end{align}

When $m=1$, \eqref{sumH} reduces to the following sum-of-Hessians equation:
\begin{align}\label{sum1}
\sigma_k(\lambda) + a\,\sigma_{k-1}(\lambda)
= \psi(x,u,\nabla u),
\end{align}
which always satisfies hypothesis (RR).
In this case, Condition \eqref{con1} is equivalent to Condition \eqref{con2}, since $\sigma_{k-1} > 0$ when $\lambda \in \Gamma_k^{n+(1)}$ (see Lemma \ref{sigmak} (5)). 
The equation \eqref{sum1} has been widely studied in \cite{LR,LR2,LYZ,JX} under the above conditions.

For general $m$, we note that Condition \eqref{con2} implies 
$\lambda \in \Gamma_k^{n+(m)}$ via \eqref{sl}, which is the first requirement of Condition \eqref{con1}. This means that 
\begin{align}\label{LRWcon}
    \tilde\Gamma_k^n = \Gamma_{k-1}^n \cap\lbrace\lambda \in \mathbb{R}^n \;|\; \sigma_k^{(n)}(\lambda)
      + \sum_{r=1}^m a_r\sigma_{k-r}^{(n)}(\lambda)>0\rbrace
\end{align}
is an admissible set. As pointed out by Li-Ren-Wang \cite{LRW2}, \eqref{LRWcon} guarantees that $F$ is elliptic under Hypothesis (RR). Otherwise, Condition \eqref{con1} implies $a_i \geqslant 0$, but is not equivalent to $\lambda \in \Gamma_{k-1}^n$ in general. Hypothesis (RR) and Condition \eqref{con2} can be regarded as generalizations of the conditions in Theorem 4 of \cite{LRW2}. Condition \eqref{con2} relaxes the convexity requirement of solutions in \cite{LRW2} to $\Gamma_{k-1}^n$, while Condition \eqref{con1} describes a different setting in which the same results may still hold.

As in the case of the $k$-Hessian equation (\cite{GRW, LRW, RW1, RW2, Zhang}), we exploit the concavity for the operator \eqref{sumH}.
Let $\lambda^{\downarrow} \in \mathbb{R}^n$ denote the decreasing rearrangement of $\lambda \in \mathbb{R}^n$, defined by
\[
\lambda^{\downarrow} := (\lambda_1, \ldots, \lambda_n), \quad \text{where } 
\lambda_1 \geqslant \cdots \geqslant \lambda_n.
\]
Motivated by the work of Zhang \cite{Zhang}, we establish a concavity inequality for the sum-of-Hessians operator $F$ in Lemma \ref{Fconcavity} for $k=n-1, n$ or $k\leqslant n-2$ under a dynamic semi-convexity condition, i.e., $(D^2 u)_{\min} \geqslant - \delta \Delta u$. This inequality can be regarded as a Jacobi inequality valid for sufficiently large $\ln u_{11}$, at
points where $u_{11}$ attains its maximum. For the 2-Hessian equation, the analogous Jacobi inequality concerning $\ln \Delta u$ was previously established by Shankar-Yuan in \cite{SY2} for $n=4$ and for $n\geqslant 5$ under the dynamic semi-convexity condition. In particular, when $k=n-1$, we extend the concavity inequality for the $(n-1)$-Hessian operator established by Ren-Wang \cite{RW1} (see also \cite{Tu}) to the general sum-of-Hessians operator, drawing on observations of Zhang \cite{Zhang} and Lu-Tsai \cite{LT}. When $m = 0$, the inequality in Lemma \ref{Fconcavity} reduces to the case in Lemma \ref{key}, which was first proved by Zhang in \cite{Zhang} under the assumptions that $\lambda^{\downarrow} \in \Gamma_k^n$ and $\lambda_n \geqslant -K$ for some constant $K > 0$. This inequality is crucial for establishing Pogorelov estimates for \eqref{eq1} and characterizing the rigidity of \eqref{sigma-k}. Namely, we have the following theorem.

\begin{theorem}\label{thm1}
Assume that $\psi \in C^2(\overline{\Omega} \times \mathbb{R} \times \mathbb{R}^n)$ satisfies $\psi > \psi_0 > 0$ for some constant $\psi_0 > 0$ in $\overline{\Omega}$. Then, for any smooth solution of the Dirichlet problem for equation \eqref{eq1} satisfying Hypothesis (RR)~and Condition \eqref{con1} (or Condition \eqref{con2}), the following estimate holds:
\[
(-u)^{\alpha} \Delta u \leqslant C
\]
under one of the following conditions:
\begin{enumerate}
    \item[(A)] $k = n-1$ or $k = n$;
    \item[(B)] $2\leqslant k \leqslant n-2$, and 
    $(D^2 u)_{\min} \geqslant - \delta \Delta u$ for some sufficiently small $\delta > 0$.
\end{enumerate}
Here $\alpha > 0$ and $C > 0$ depend only on $n$, $k$, $m$, $\delta$, $\psi_0$, $||a||$, $||\psi||_{C^2}$, and $||u||_{C^1}$, and are independent of $\Omega$.
\end{theorem}

When $m = 1$, Condition \eqref{con1} is equivalent to Condition \eqref{con2}. 
The Pogorelov estimates in case (B) were obtained by Li-Ren \cite{LR} under the additional assumption that $\sigma_k > 0$, together with Condition \eqref{con2}. 
In contrast, our case (B) relaxes the assumption in \cite{LR} due to Lemma \ref{sigmak} (8).
When $m = 1$, in case (A) with $k = n$, the same Pogorelov estimates were first established by Li-Ren \cite{LR}; for completeness, we provide an alternative proof.
When $m = 1$ and $k = 2$ or $k = 3$, Liu-Ren \cite{LR2} derived the same Pogorelov estimates under Condition \eqref{con2};
When $m = 1$, for general $k$, assuming in addition that the solution is semi-convex, Liang-Yan-Zhu \cite{LYZ} and Jin-Xu \cite{JX} also proved the Pogorelov estimates.

We also consider the Dirichlet problem with $\psi$ depending only on $x$. In this case, we establish the following Pogorelov estimates, which are independent of $||\nabla u||_{C^0}$.

\begin{theorem}\label{thm2}
	Suppose that $u : \Omega \rightarrow \mathbb{R}$ is a smooth admissible solution to the Dirichlet problem \eqref{eq1} 
    with $\psi \equiv \psi(x)$. Suppose that the sum-of-Hessians operator $F$ satisfies Hypothesis (RR) and Condition \eqref{con1} (or Condition \eqref{con2}). Then, we have the  following interior estimate
	\begin{equation}\label{pogrelov}
		(-u)^\beta \Delta u \leqslant C,
	\end{equation}
    under one of the following conditions,
\begin{enumerate}
\item[(A)] $k=n-1$ or $k=n$,
     \item[(B)] $2\leqslant k\leqslant n-2$, and
    $(D^2 u)_{\min}\geqslant - \delta \Delta u$ for some sufficiently small $\delta >0$.
\end{enumerate}
Here $\beta> 0$ and $C > 0$ depend only on $n$, $k$, $m$, $\delta$, $\psi_0$, $||a||$, $||\psi||_{C^2}$, $||u||_{C^0}$, and the diameter of the domain $\Omega$.
\end{theorem}

For Pogorelov estimates for $k$-Hessian equations and $k$-th mean curvature equations, with constants $C$ depending on $\|\nabla u\|_{C^0}$, see Chou-Wang \cite{CW} and Sheng-Urbas-Wang \cite{SUW}, respectively, in the case where $\psi$ does not depend on $\nabla u$. 
We also refer the reader to Liu-Trudinger \cite{LiuT} and Jiang-Trudinger \cite{JT} for related estimates for the Monge-Amp\`ere equation. Recently,
Pogorelov estimates for Hessian quotient equations were established by Lu-Tsai \cite{LT2}.
The estimates in Theorem \ref{thm2} play a crucial role in establishing a rigidity theorem for entire solutions in $\mathbb{R}^n$ with quadratic growth. 
Notably, the condition $(D^2 u)_{\min} \geqslant - \delta \Delta u$ is weaker, since it allows $-(D^2 u)_{\min}$ to be sufficiently large. 
This relaxation enables us to extend several rigidity theorems obtained in earlier works \cite{BCGJ, CD, LRW, Zhang}.
More precisely, we prove a Liouville-type theorem for smooth entire solutions in $\mathbb{R}^n$ to the equation \eqref{sigma-k}.
A function $u:\mathbb{R}^n \rightarrow \mathbb{R}$ is said to satisfy the \textit{quadratic growth condition} if there exist positive constants $C_1$, $C_2$, and $R_0$ such that
\begin{align}
\label{qg}
u(x) \geqslant C_1 |x|^2 - C_2, \quad \text{for } |x| \geqslant R_0.
\end{align}

As a direct corollary of Theorem \ref{thm2}, we obtain the following Liouville-type results under the quadratic growth condition. 
The proof can be found in \cite{BCGJ, LRW} and is therefore omitted.

\begin{corollary}\label{cor1}
	For $2\leqslant k\leqslant n-2$, suppose $u:\mathbb{R}^n \rightarrow \mathbb{R}$ is a smooth solution to equation \eqref{sigma-k} under the Hypothesis (RR) and Condition \eqref{con1} (or Condition \eqref{con2}), satisfying a dynamic semi-convex condition, i.e., $(D^2 u)_{\min}\geqslant -\delta\Delta u$ for some sufficiently small $\delta > 0$. If $u$ satisfies the quadratic growth condition \eqref{qg}, then $u$ must be a quadratic polynomial.
\end{corollary}
\begin{remark}
    For the case $k>2$, $m=0$, this can be seen as a generalization of the rigidity results recently derived by Shankar-Yuan for the 2-Hessian equation, as an application of their Theorem 1.2 in \cite{SY2}. For the case $m=1$ with Condition \eqref{con2}, Liu-Ren \cite{LR2} proved the same rigidity results under the assumption $\s_k>0$.
\end{remark}
In particular, we derive the following rigidity results for admissible entire solutions of equation \eqref{sigma-k}.
\begin{corollary}\label{cor2}
	For $k=n-1$ or $k=n$, suppose $u:\mathbb{R}^n \rightarrow \mathbb{R}$ is a smooth solution to equation \eqref{sigma-k} under the Hypothesis (RR) and Condition \eqref{con1} (or Condition \eqref{con2}). If $u$ satisfies the quadratic growth condition \eqref{qg}, then $u$ must be a quadratic polynomial.
\end{corollary}

For the case $m=0$, \eqref{sumH} reduces to equation \eqref{khessian}. Under the quadratic growth condition, Chen and Xiang \cite{CX}
proved the theorem for the case $k=2$, with solutions in the $\Gamma_k^n$ cone, under the additional assumption that $\sigma_3 (\lambda) > -K$ for some positive constant $K$. This assumption was previously employed by Guan and Qiu \cite{GQ} to establish a purely interior second order estimate. Chu and Dinew \cite{CD} recovered this result for general $k$ by assuming $\sigma_{k+1}(\lambda) > -K$. This was achieved as a special case of their broader framework, which encompasses a wide class of fully nonlinear equations and, in particular, includes the $p$-Monge-Amp\`ere equations. Du \cite{Du} established a necessary and sufficient condition under which a locally strictly convex solution to equation \eqref{sigma-k} must be a quadratic function. It should be noted that when $k=2$ and $n=3$ Theorem \ref{thm2} was proved by Chen and Xiang \cite{CX} under the assumption that $\lambda(D^2u) \in \Gamma_{2}^3$; when $k=n-1$, Theorem \ref{thm2} was proved by Tu \cite{Tu} under the assumption that $\lambda(D^2u) \in \Gamma_{n-1}^n$, using ideas from \cite{Zhang} and \cite{LT}. Therefore, the problem remains open for the case $2\leqslant k\leqslant n-2$.

The structure of the paper is as follows. 
In Section 2, we introduce some useful notation, properties of the $k$-th elementary symmetric function, and some preliminary calculations and estimates. 
In Section 3, we prove the key inequality for sums of elementary symmetric functions. 
In Sections 4 and 5, we prove Theorems~\ref{thm1} and~\ref{thm2} for Case (B) under Condition~\eqref{con1}, respectively. 
In Sections 6 and 7, we discuss Case (A) and Condition~\eqref{con2}, respectively.

\section{preliminaries}

	For $x = (x_1, \ldots, x_{n})\in \mathbb{R}^{n}$,
	we use $s_{k;i} (x)$ to denote the $k$-th elementary function with $x_i = 0$
	and $s_{k;ij} (x)$ to denote the $k$-th elementary function with $x_i = x_j = 0$.
	Define $x | j : = (x_1, \ldots, x_{j-1}, 0, x_{j+1}, \ldots, x_{n} )$. Then, $s_{k;i} = s_k(x|i)$ and $s_{k;ij} = s_k (x|ij)$.
	Below, we list some basic properties of the $k$-th elementary symmetric function.
	\begin{lemma}\label{sigmak}
		For $x = (x_1, \ldots, x_{n}) \in \mathbb{R}^{n}$ and $k = 1, \cdots, n$, we have
		\begin{itemize}
			\item[(1)] $ s_k = s_{k;i} + x_i s_{k-1; i}$, $\forall$ $1 \leqslant i \leqslant n$;
			\item[(2)] $ \sum_{i=1}^{n} s_{k;i} = (n-k) s_k$ and $\sum_{i=1}^{n} s_{k-1;i} x_i = k s_k$;
			\item[(3)] $\sum_{i=1}^{n} s_{k-1;i} x_i^2 = s_1 s_k - (k+1) s_{k+1}$;
			\item[(4)] If $x \in \Gamma_k^n$ and $x_1 \geqslant \cdots \geqslant x_{n}$, then we have $x_1 s_{k-1; 1} \geqslant \frac{k}{n} s_k$;
			\item[(5)] If $x \in \Gamma_{k}^n$, then we have $x | j \in \Gamma_{k-1}^n$ for all $1\leqslant j \leqslant n$;
			\item[(6)] If $x \in \Gamma_{k}^n$ and $x_i \geqslant x_j$, then we have $s_{k-1;i} \leqslant s_{k-1;j}$;
			\item[(7)] If $x \in \Gamma_k^n$ and $x_1 \geqslant \cdots \geqslant x_{n}$, then $s_{k-1} \geqslant x_1 \cdots x_{k-1}$;
			\item[(8)] For  $x \in \Gamma_{k}^n$, if  $s_{k} \leqslant A_{1}$  and  $s_{k+1} \geqslant-A_{2}$  for some constants  $A_{1}$, $A_{2} \geqslant 0$, we have that  for all $i$, $x_{i}+K \geqslant 0$  for some uniform positive constant  $K$  depending on  $n, k, A_{1}$ and $A_{2}$;
            \item[(9)] For $x \in \Gamma_k^n$ with $x_1 \geqslant \cdots \geqslant x_{n}$, we have
            $s_{k-1;k} \geqslant C s_{k-1}$ for some $C$ depending only on $n$ and $k$.
		\end{itemize}
	\end{lemma}
	
	\begin{proof}
		For (1), (2) and (3), it is trivial. 
		For (4), see \cite{HMW}. For (5), see \cite{HS}. It can also be derived from another characterization of $\Gamma_k^n$ given as below (see \cite{K})
		\[
		\Gamma_k^{n} = \left\{x \in \mathbb{R}^n \Big| 
        \begin{aligned}
        & s_k > 0,\frac{\partial s_k}{\partial x_{i_1}} > 0, \ldots, 
		\frac{\partial^k s_k }{\partial x_{i_1} \cdots \partial x_{i_k}} > 0,\\
        &\quad\quad \mbox{for all} \; 1 \leqslant i_1 < \cdots < i_k \leqslant n \end{aligned}
        \right\}.
		\]
		For (6) and (7), see \cite{YYLi}.
		For (8), see \cite{LRW, Zhang}.
        For (9), see \cite{LTr}.
	\end{proof}
	
\begin{lemma}\label{key-c}
	Suppose $(a_1,\dots,a_m)\in\mathbb R^m$ and $y_i\in \mathbb R, i=1,\dots,m$ are the roots of an $m$-th degree polynomial $P(t)=t^m+\sum_{r=1}^m(-1)^ra_rt^{m-r}$, then $\forall \la\in \mathbb{R}^n$, we have 
    \begin{equation}
    \label{eq:linear-combination}
		\sigma_k^{(n+m)}(\lambda,y)
				=\sigma_k^{(n)}(\lambda)+a_1\sigma_{k-1}^{(n)}(\lambda)+\cdots+a_m\sigma_{k-m}^{(n)}(\lambda).
			\end{equation}
	\end{lemma}
		
    \begin{proof}
			Let $\lambda=(\lambda_1,\dots,\lambda_n)\in\mathbb R^n$ and $y=(y_1,\dots,y_m)\in\mathbb R^m$.  
			Set $x=(\lambda,y)\in\mathbb R^{n+m}$.
			For $r\ge 0$, let $e_r(y)$ denote the $r$-th elementary
			symmetric polynomial in $y_1,\dots,y_m$, with the convention $e_0\equiv 1$.
			Then for every integer $k\ge 0$, by the definition of $\s_k$, we have
			\begin{equation}\label{eq:lifting-identity}
				\sigma_k^{(n+m)}(x)\;=\;\sum_{r=0}^{m} e_r(y)\,\sigma_{k-r}^{(n)}(\lambda).
			\end{equation}
				By Vi\`ete's formulas,
			\[
				e_r(y)=a_r,\qquad r=1,\dots,m,
			\]
			which yields
            \eqref{eq:linear-combination} combining \eqref{eq:lifting-identity}.		
		\end{proof}
	
	Next, we recall some previous lemmas from \cite{HZ, Zhang}.
	\begin{lemma}\label{xk}
		Suppose $x^{\downarrow} \in \Gamma_k^n$. Then, we have
		\[
		x_k \leqslant C \left(\s_k^{\frac{1}{k}} (x) + |x_{n}|\right)
		\]
		for some constant $C$ depending on $n$ and $k$.
	\end{lemma}
	
	\begin{proof}
		
		Since $x_k + \cdots + x_{n} > 0$, we have $|x_{n}| \leqslant (n-k+1) x_k$, and hence
		\begin{equation}\label{c}
			x_1 \cdots x_{k-1} x_k \leqslant \s_k + C (n, k) x_1 \cdots x_{k-1} |x_{n}|.  
		\end{equation}
		Therefore, we obtain
		\begin{equation}
			\label{c1}
			x_k \leqslant 2C(n,k) \left(\s_k^{\frac{1}{k}}+|x_{n}|\right).
		\end{equation}
		This completes the proof of the lemma.
	\end{proof}

Set
	\[
	q_k = \frac{\s_k}{\s_{k-1}}, \; \s_{k;i} = \frac{\partial \s_{k+1}}{\partial \lambda_i} \; \mbox{and} \; q_{k;i} = \frac{\s_{k;i}}{\s_{k-1;i}}.
	\] 
	Note that $q_1 = \s_1$ since $\s_0 \equiv 1$.
	Recall the following lemmas.

\begin{lemma}\label{HS1}
If $\lambda \in \Gamma_k^n$, then for any $\xi\in \mathbb{R}^n$,
\begin{align*}
   -\partial^2_\xi q_{2;i_1,\cdots,i_{k-2}}\geqslant \frac{|[\xi]_{i_1,\cdots,i_{k-2}}^\perp|^2}{\s_1(\lambda|i_1,\cdots,i_{k-2})}.
\end{align*}
Here for any $1\leqslant \ell\leqslant k-2$, $[\xi]_{i_1,\cdots,i_{\ell}}$ denotes the vector that is obtained by throwing out the $i_1$-th,$\cdots,$ $i_{\ell}$-th components of $\xi$  and $[\xi]_{i_1,\cdots,i_{\ell}}^\perp$ denotes orthogonal part of $[\xi]_{i_1,\cdots,i_{\ell}}$ with respect to $[\lambda]_{i_1,\cdots,i_{\ell}}.$
Moreover,
\begin{align*}
       -\partial^2_{\xi}q_{k-1;i_1}\geqslant -\sum_{i\in\{1,\cdots, n\}\setminus \{i_1\}}\frac{\lambda_i^2\partial^2_{\xi}q_{k-2;i_1,i}}{(k-1)(q_{k-1;i_1,i}+\lambda_i)^2}.
   \end{align*}   
\end{lemma}

\begin{lemma}
If $\lambda \in \Gamma_k^n$, then for any $\zeta\in \mathbb{R}^n$,
    \begin{align*}
       -\s_{k-1}\partial^2_\zeta q_k\geqslant C\sum_{j=1}^{k-1}\lambda_1\cdots\hat{\lambda}_j\cdots\lambda_{k-1}\left|[\zeta]^\perp_{1,\cdots,\hat{j},\cdots,k-1}\right|^2.
   \end{align*}
\end{lemma}
It follows from the above two lemmas that 
\begin{lemma}\label{decom}
Assume that $\zeta=(0,\zeta_2,\cdots,\zeta_{k-1},\zeta_k,\dots,\zeta_n)$ satisfying 
\begin{align*}
(\zeta_k,\dots,\zeta_n)\perp(\lambda_k,\dots,\lambda_n).
\end{align*}
If $\lambda \in \Gamma_k^n$, then
    \begin{align*}
       -\s_{k-1}\partial^2_\zeta q_k\geqslant C\frac{\sum_{j=1}^{k-1}\lambda_1\lambda_2\cdots\lambda_{k-1}\lambda_k^2\zeta_j^2}{\lambda_j^3}+C\lambda_1\lambda_2\cdots\lambda_{k-2}\sum_{p=k}^{n}\zeta_p^2.
   \end{align*}
\end{lemma}

\begin{lemma}\label{qk,con}
\begin{align*}
   \frac{\frac{\partial^2 q_k}{\partial\lambda_p\partial\lambda_q }}{q_k}=\frac{\frac{\partial^2 \s_k}{\partial\lambda_p\partial\lambda_q}}{\s_k}-\frac{\frac{\partial \s_k}{\partial\lambda_p}}{\s_k}\frac{\frac{\partial \s_{k-1}}{\partial\lambda_q}}{\s_{k-1}}-\frac{\frac{\partial \s_k}{\partial\lambda_q}}{\s_k}\frac{\frac{\partial \s_{k-1}}{\partial\lambda_p}}{\s_{k-1}}-\frac{\frac{\partial^2 \s_{k-1}}{\partial\lambda_p\partial\lambda_q}}{\s_{k-1}}+2\frac{\frac{\partial \s_{k-1}}{\partial\lambda_p}\frac{\partial \s_{k-1}}{\partial\lambda_q}}{\s_{k-1}^2}.
\end{align*}
\end{lemma}

\section{A Key Concavity Lemma}

In this section, we establish a concavity inequality for sums of Hessian operators.
We begin by proving the following key inequality for the $k$-Hessian operator.
Corresponding concavity inequalities for the sum operators follow as direct corollaries.
	\begin{lemma}\label{key}
    There exist a sufficiently small \(\delta\in(0,1)\) and a sufficiently large \(K_0\), depending only on \(n\) and \(k\), such that if
    $\lambda\in \Gamma^n_k$ and \[\lambda_1\geq \lambda_2\geq \cdots\geq \lambda_n \geqslant -\delta \lambda_1,\quad\lambda_1 /\s_k^{\frac{1}{k}}>K_0,\]
		then the following inequality 
		\begin{equation}
			\label{positive}
			- \sum_{p\neq q} \frac{\s_k^{pp,qq} \xi_p \xi_q}{\s_k}
			+ 2 \frac{(\sum_i \s_k^{ii} \xi_i)^2}{\s_k^2} 
			+ 2 \sum_{i > 1} \frac{\s_k^{ii} \xi_i^2}{(1 + 2\delta )\lambda_1 \s_k} 
			\geqslant (1+\gamma) \frac{\s_k^{11} \xi_1^2}{\lambda_1 \s_k}
		\end{equation}
        holds for every vector $\xi = (\xi_1, \ldots, \xi_n) \in \mathbb{R}^{n}$, with 
        $\gamma = 1/(1+16k)$.
	\end{lemma}

	\begin{proof}
        Let $$\tl_i=\frac{\la_i}{\la_1}, \quad 1\leqslant i \leqslant n.$$
		As the inequality is homogeneous of degree $-2$, it is equivalent to verify that 
		\[
		- \sum_{p\neq q} \frac{\s_k^{pp,qq}(\tl)\xi_p \xi_q}{\s_k(\tl)}
		+ 2 \frac{(\sum_i \s_k^{ii}(\tl) \xi_i)^2}{\s_k(\tl)^2} 
		+ 2 \sum_{i > 1} \frac{\s_k^{ii}(\tl) \xi_i^2}{(1 + 2\delta )\s_k(\tl)} 
		\geqslant (1+\gamma) \frac{\s_k^{11}(\tl) \xi_1^2}{\s_k(\tl)},
		\]
		Note that
		\[
		\tl_1=1, \; \tl_n\geqslant -\de \; \mbox{and}\; 
		\s_k(\tl)=\frac{\s_k(\la)}{\la_1^k}<K_0^{-k}.
		\]
		We obtain that
		\begin{equation}
			\label{main}
			\begin{aligned}
				&- \sum_{p\neq q} \frac{\s_k^{pp,qq} \xi_p \xi_q}{\s_k}
				+ 2 \frac{(\sum_i \s_k^{ii} \xi_i)^2}{\s_k^2}\\
				= &\; - \frac{\partial_\xi^2 q_k}{q_k} + (\partial_\xi \log q_k)^2
				- \partial_\xi^2 \log \s_{k-1} + (\partial_\xi \log\s_k)^2\\
				\geqslant &\; - \frac{\partial_\xi^2 q_k}{q_k} + \frac{1}{2} (\partial_\xi \log \s_{k-1})^2 - \partial_\xi^2 \log \s_{k-1}.
			\end{aligned}
		\end{equation}
		Here we used Cauchy-Schwarz inequality
		\[\begin{aligned}
			( \partial_\xi \log q_k )^2 = &\; ( \partial_\xi \log \s_k - \partial_\xi \log \s_{k-1} )^2 \\
			\geqslant &\; \frac{1}{2} (\partial_\xi \log \s_{k-1})^2 - (\partial_\xi \log \s_{k})^2.
		\end{aligned}\]
		Define
		\[
		\hat \xi = (0, \xi_2, \cdots, \xi_{n}).
		\]
		By $\partial_\xi = \xi_1 \partial_1 + \partial_{\hat \xi}$,
		a direct calculation shows that
		\begin{equation}\label{f1}
			\partial_\xi \log \s_{k-1} = \frac{\s_{k-2;1} \xi_1}{\s_{k-1}} + \partial_{\hat \xi} \log \s_{k-1}    
		\end{equation}
		and
		\begin{equation}\label{f2}
          \begin{aligned}
			\partial_\xi^2 \log \s_{k-1} = &\; \partial_{\hat \xi}^2 \log \s_{k-1} 
			- \Big( \frac{\s_{k-2; 1} \xi_1}{\s_{k-1}} \Big)^2 \\
			&\; + 2 \sum_{j>1} \frac{\s_{k-1} \s_{k-3;1j} - \s_{k-2;1} \s_{k-2;j}}{\s_{k-1}^2} \xi_1 \xi_j. 
            \end{aligned}
		\end{equation}
		We note that \begin{equation}\label{concavity}
			- \partial_{\hat \xi}^2 \log \s_{k-1} = - \sum_{i=1}^{k-1} \frac{\partial_{\hat \xi}^2 q_i}{q_i} + \sum_{i=1}^{k-1} (\partial_{\hat \xi} \log q_i)^2 \geqslant \sum_{i=1}^{k-1} (\partial_{\hat \xi} \log q_i)^2.
		\end{equation}
		By \eqref{f1}, \eqref{f2} and \eqref{concavity},
		we get
		\[\begin{aligned}
			& \frac{1}{2} (\partial_\xi \log \s_{k-1})^2 - \partial_\xi^2 \log \s_{k-1}\\
			\geqslant &\; \frac{1}{2} \Big(\frac{\s_{k-2;1} \xi_1}{\s_{k-1}} + \sum_{i=1}^{k-1} \partial_{\hat \xi} \log q_i\Big)^2 + \sum_{i=1}^{k-1} (\partial_{\hat \xi} \log q_i)^2\\
			&\; +\Big( \frac{\s_{k-2; 1} \xi_1}{\s_{k-1}} \Big)^2
			- 2 \sum_{j>1} \frac{\s_{k-1} \s_{k-3;1j} - \s_{k-2;1} \s_{k-2;j}}{\s_{k-1}^2} \xi_1 \xi_j.
		\end{aligned}
		\]
		By Cauchy-Schwarz inequality, we see that
		\[\begin{aligned}
			\Big(\frac{\s_{k-2;1} \xi_1}{\s_{k-1}} + \sum_{i=1}^{k-1} \partial_{\hat \xi} \log q_i\Big)^2
			\geqslant \frac{1}{k} \Big(\frac{\s_{k-2;1} \xi_1}{\s_{k-1}} \Big)^2
			- \sum_{i=1}^{k-1} (\partial_{\hat \xi} \log q_i)^2.
		\end{aligned}\]
		Let
		\[
		I_{1j} : = \frac{\s_{k-1} \s_{k-3;1j} - \s_{k-2;1} \s_{k-2;j}}{\s_{k-1}^2}. 
		\]
		We then obtain from the above two inequalities and \eqref{main} that
			\begin{equation}
				\label{main-1}
				\begin{aligned}
					& - \sum_{p\neq q} \frac{\s_k^{pp,qq} \xi_p \xi_q}{\s_k}
					+ 2 \frac{(\sum_i \s_k^{ii} \xi_i)^2}{\s_k^2}\\
					\geqslant &\; - \frac{\partial_\xi^2 q_k}{q_k} + \Big( 1 + \frac{1}{2k} \Big) \Big(\frac{\s_{k-2;1} \xi_1}{\s_{k-1}}\Big)^2 - 2 \sum_{j>1} I_{1j} \xi_1 \xi_j.
				\end{aligned}
			\end{equation}
			
			Now we deal with the last two terms in the above inequality. By the assumption we note that $-\tilde \lambda_n$ is not larger than $\delta\ll 1$. From Lemma \ref{xk}, we have 
            \begin{align}\label{tl k}
                \tl_k\leqslant C'(\de+\la_1^{-1}\s_k^{\frac{1}{k}}(\la))\leqslant C\de'
            \end{align}
            where $\de'=2\max\lbrace \de,K_0^{-1}\rbrace.$
			Then we have
			\[
			|\s_{k-1;1}|\leqslant C\tl_2 \cdots \tl_{k-1}\tl_k  \leqslant C\delta' \s_{k-1},
			\]
			
			By Lemma \ref{sigmak} (1), we obtain that
			\begin{equation}
				\label{s-k-2-1}
				\frac{\s_{k-2;1}}{\s_{k-1}} = \frac{\s_{k-1} - \s_{k-1;1}}{\tl_1 \s_{k-1}} \geqslant \frac{1 - C \delta'}{\tl_1},
			\end{equation}
			where $C$ depends on $n$ and $k$. From Lemma \ref{sigmak} (1), we can derive that
			\[
			\s_{k-2;1} \s_{k-2;j} - \s_{k-1} \s_{k-3;1j} 
			= \s_{k-2;1j}^2 - \s_{k-1;1j} \s_{k-3;1j} \geqslant 0,
			\] 
			where we used Newton's inequality.
			For $1 < j \leqslant k-1$,  since $|\tl_{n}| \leqslant (n-k) \tl_k$ we have
			\[
			\s_{k-2;1j} 
			\leqslant C \frac{\tl_1 \cdots \tl_k}{\tl_1 \tl_j}
			\leqslant C \frac{\s_{k-1} \tl_k}{\tl_j}
			\]
			and
			\[
			|\s_{k-1;1j}| \s_{k-3;1j} \leqslant C \frac{\tl_1 \cdots \tl_k^2}{\tl_1 \tl_j} 
			\frac{\tl_1 \cdots \tl_{k-1}}{\tl_1 \tl_j} \leqslant C \Big(\frac{\s_{k-1} \tl_k}{\tl_j}\Big)^2.
			\]
			On the other hand, for $k \leqslant j \leqslant n$, we have 
			\[
			\s_{k-2;1j} \leqslant C \frac{\tl_1 \cdots \tl_{k-1}}{\tl_1} \leqslant C \s_{k-1}
			\]
			and
			\[
			|\s_{k-1;1j}| \s_{k-3;1j} \leqslant C \frac{\tl_1 \cdots \tl_k}{\tl_1} 
			\frac{\tl_1 \cdots \tl_{k-2}}{\tl_1} \leqslant C \s_{k-1}^2.
			\]
			Therefore, we get, for $1 < j \leqslant k-1$, that
	\begin{align}\label{I1j 1}
	    - 2 I_{1j} \xi_1 \xi_j
			\geqslant -2 \frac{C \tl_k^2}{\tl_j^2} |\xi_1 \xi_j|
			\geqslant - \varepsilon \xi_1^2 - \frac{C\tl_k^4 \xi_j^2}{\varepsilon  \tl_j^4}\geqslant - \varepsilon \xi_1^2 - \frac{C\tl_k \xi_j^2}{\varepsilon  \tl_j},
	\end{align}
and, for $k \leqslant j \leqslant n$, that
\begin{align}\label{I1j 2}
			- 2 I_{1j} \xi_1 \xi_j
			\geqslant -2 C |\xi_1 \xi_j|
			\geqslant - \varepsilon \xi_1^2 - \frac{C \xi_j^2}{\varepsilon },
\end{align}
			where we have used the Cauchy-Schwarz inequality. Here $\varepsilon > 0$ is a small constant to be decided later and $C$ depends only on $n$ and $k$.

			If $\s_{k;j}>0$, then $\s_{k-1;1j}> 0$ and by Lemma \ref{sigmak} (1), we see that 
			\begin{equation}
				\label{sk-1}
				\s_{k-1;j} = \s_{k-2;1j} + \s_{k-1;1j} > 0\quad \mbox{for} \; 
				\tl \in \Gamma_k.  
			\end{equation}
			By Maclaurin's inequality, we have
			\[
			(\s_{k-1;1j})^{\frac{k-2}{k-1}} \leqslant C(n,k)\s_{k-2;1j}.
			\]
			On the other hand, we have
			\[
			\s_{k-1;1j} \leqslant C \tl_2\cdots \tl_{k-1} \tl_k \leqslant C \delta'.
			\]
            where we use $0<\tl_i\leqslant 1$ for any $2\leqslant i\leqslant k$ and \eqref{tl k} in the last inequality. Therefore, we obtain
			\begin{equation}
				\label{ineq-2}
				0 < \s_{k-1;1j} \leqslant C\delta'^{\frac{1}{k-1}} \s_{k-2;1j},
			\end{equation}
			from which we derive that
			\[
			\s_{k-1;1} = \tl_j\s_{k-2;1j} + \s_{k-1;1j}
			\leqslant \Big( \tl_j + C\delta'^{\frac{1}{k-1}} \Big) \s_{k-2;1j}.
			\]
			When $\tl_j\geqslant \de'^{\frac{1}{k-1}}$, combining \eqref{sk-1},
			we obtain that
			\[
			\s_{k-1;1} \leqslant C\tl_j\s_{k-1;j}.
			\]
		When $\tl_j< \de'^{\frac{1}{k-1}} $,
			we have
			\[
			\s_{k-1;1} \leqslant  C\delta'^{\frac{1}{k-1}}  \s_{k-1;j}.
			\]
            
			Combining Lemma \ref{sigmak} (4), we obtain that,
			if
			$\tl_j\geqslant \de'^{\frac{1}{k-1}}$, then $1\leq j< k$, and 
			\begin{equation}
				\label{I-j-1}
				\frac{\s_{k-1;j}}{\s_k}\geqslant \frac{C}{\tl_j};
			\end{equation}
	if $\tl_j<\de'^{\frac{1}{k-1}}$, then
			\begin{equation}
				\label{I-j-2}
				\frac{\s_{k-1;j}}{\s_k}\geqslant \frac{C}{\de'^{\frac{1}{k-1}}}.
			\end{equation}
			If $\s_{k;j}\leqslant 0$, then by Lemma \ref{sigmak} (1), we have $\s_{k-1;j}\tl_j\geqslant \s_k,$
			that is, 
			\begin{equation}
				\label{I-j-3}
				\frac{\s_{k-1;j}}{\s_k} \geqslant \frac{1}{\tl_j}.
			\end{equation}
			
			Consequently, inserting \eqref{I-j-1}, \eqref{I-j-2} and \eqref{I-j-3} into \eqref{I1j 1} and \eqref{I1j 2}, we have
			\begin{align}\label{I1j 3}
			- 2 \sum_{j>1} I_{1j} \xi_1 \xi_j \geqslant - n \varepsilon \xi_1^2 - \frac{C\delta'^{\frac{1}{k-1}}}{\varepsilon} 
			\sum_{i > 1} \frac{\s_k^{ii} \xi_i^2}{ \s_k}
			\end{align}
  where we also use \eqref{tl k} and the fact, $\tl_k\leq \tl_j$ for any $1< j\leqslant k-1$.
			Combining \eqref{main-1}, \eqref{s-k-2-1} and \eqref{I1j 3},
			we arrive at the conclusion that 
			\begin{equation}
				\label{main-2}
				\begin{aligned}
					& - \sum_{p\neq q} \frac{\s_k^{pp,qq} \xi_p \xi_q}{\s_k}
					+ 2 \frac{(\sum_i \s_k^{ii} \xi_i)^2}{\s_k^2} 
					+ 2 \sum_{i > 1} \frac{\s_k^{ii} \xi_i^2}{(1 + 2\delta ) \s_k} \\
					\geqslant &\; - \frac{\partial_\xi^2 q_k}{q_k} 
					+ \Big( 1 + \frac{1}{2k} \Big) \Big(\frac{\s_{k-2;1} \xi_1}{\s_{k-1}}\Big)^2 
					- n \varepsilon \xi_1^2
					+ \Big(\frac{2}{1+2\delta} - \frac{C\delta'^{\frac{1}{k-1}}}{\varepsilon} \Big)
					\sum_{i > 1} \frac{\s_k^{ii} \xi_i^2}{\s_k}\\
					\geqslant &\; - \frac{\partial_\xi^2 q_k}{q_k} + \Big( 1 + \frac{1}{2k} 
					- 2n \varepsilon \Big) \Big(1-C\delta'\Big)^2 \xi_1^2
					+ \Big( 2 - 4\delta - \frac{C\delta'^{\frac{1}{k-1}}}{\varepsilon} \Big)
					\sum_{i > 1} \frac{\s_k^{ii} \xi_i^2}{\s_k}.
				\end{aligned}
			\end{equation}
			We denote by $Q$ the following quadratic form:
			\[
			Q:=  - \sum_{p\neq q} \frac{\s_k^{pp,qq} \xi_p \xi_q}{\s_k}
			+ 2 \frac{(\sum_i \s_k^{ii} \xi_i)^2}{\s_k^2} 
			+ 2 \sum_{i > 1} \frac{\s_k^{ii} \xi_i^2}{(1 + 2\delta ) \s_k}.
			\]
			We further derive from \eqref{main-2} that, for $\varepsilon = \frac{1}{8(n+1)k}$, there exists a constant $C$ depending only on $n$ and $k$, such that
			\begin{equation}
				\label{main-3}
				\begin{aligned}
					Q 
					\geqslant &- \frac{\partial_\xi^2 q_k}{q_k} + \Big( 1 + \frac{1}{4k} \Big) \Big(1-C\delta'\Big)^2 \xi_1^2
					+ \Big( 2 - C\delta'^{\frac{1}{k-1}} \Big)
					\sum_{i > 1} \frac{\s_k^{ii} \xi_i^2}{ \s_k}.
				\end{aligned}
			\end{equation}
			For sufficiently small $\delta$ depending only on $n,\ k$,
			we have 
			\begin{equation}
				\label{main-4}
				\begin{aligned}
					Q \geqslant - \frac{\partial_\xi^2 q_k}{q_k} +\Big( 1 + \frac{1}{8k} \Big) \xi_1^2
					+ \Big( 2 - C\delta'^{\frac{1}{k-1}} \Big)
					\sum_{i > 1} \frac{\s_k^{ii} \xi_i^2}{\s_k}.
				\end{aligned}
			\end{equation}
			We now assume that
			\[
			\s_{k;1}> - \frac{\s_k}{16k}.
			\]
			Then, from $\s_k = \tl_1 \s_{k-1;1} + \s_{k;1}$, it follows that 
			\[
			1 + \frac{1}{16k}>\frac{\s_{k-1;1}}{\s_k}.
			\]
			Therefore, for sufficiently small $\delta$ depending only on $n,\ k$, we obtain that
			\begin{equation}
				\label{main-5}
				\begin{aligned}
					Q \geqslant \Big( 1 + \frac{1}{1+ 16k} \Big) \frac{\s_{k}^{11} \xi_1^2}{ \s_k},
				\end{aligned}
			\end{equation}
			which is exactly the inequality stated in \eqref{positive} with $\gamma = 1/(1+ 16k)$.
			
			Next, we consider the case where
			\begin{equation}
				\label{sit-2}
				\s_{k;1} \leqslant - \frac{\s_k}{16k}.
			\end{equation}
			To proceed, we first establish several preliminary claims.
			
			\textit{\textbf{Claim 1}:} We have
			\begin{equation}
				\label{z-2'}
				\tl_k \leqslant C |\tl_{n}|\leqslant C\de
			\end{equation}
			for some constant $C$ depending on $n$ and $k$.
			Recall Lemma \ref{xk} and \eqref{c}. Since
			\begin{equation}
				\label{z-3}
				\frac{1}{16k} \s_k (\tl) < - \s_{k;1} (\tl)
				\leqslant C \tl_2 \cdots \tl_{k-1} \tl_k^2,
			\end{equation}
			we have 
			\[\tl_1\tl_2 \cdots \tl_k\leqslant C \tl_2 \cdots \tl_k^2 + C\tl_1 \cdots \tl_{k-1} |\tl_n|,\]
        Then due to \eqref{tl k}, we obtain 
        \begin{align*}
            \tl_k\leqslant C|\tl_n|\leqslant C\de,
        \end{align*}
			which implies Claim 1.
			
			\textit{\textbf{Claim 2}:} There exist two positive constants $c_1$ and $c_2$ only depending on $n$ and $k$, such that
			\begin{equation}
				\label{claim-2}
				-\s_{k+1} (\tl) \geqslant c_1 \tl_1 \cdots \tl_{k-1} \tl_k^2 \geqslant c_2\s_k (\tl).
			\end{equation}
			Note that
			\eqref{sit-2} implies that
			$\tl_{n} < 0$.
			Since
			$\s_{k;l} = \s_k - \tl_l \s_{k-1;l} > 0$ if $\tl_l < 0$,
			we know $\tl | l \in \Gamma_k$.
			Thus, by $\tl|n \in \Gamma_k$, we have the following:
			\begin{equation}
				\label{s_k-1n}
				\tl_1 \cdots \tl_{k-1} \tl_{n}^2 \leqslant \s_{k-1;n} \tl_{n}^2.
			\end{equation}
			By Lemma \ref{sigmak} (3), we obtain that
			\begin{equation}
				\label{z-1}
				\begin{aligned}
					\tl_1 \cdots \tl_{k-1} \tl_{n}^2<\sum_{1 \leqslant j \leqslant {n}} \s_{k-1;j} \tl_j^2 
					= \s_1 \s_k
					- (k + 1) \s_{k+1}.
				\end{aligned}
			\end{equation}
			Furthermore, by \eqref{sit-2}
			\begin{equation}
				\frac{1}{16k} \s_k<-\s_{k;1}=\s_{k+1;1}-\s_{k+1}.
			\end{equation}
			Then we have 
			\[
			\tl_1 \cdots \tl_{k-1} \tl_{n}^2\leqslant 
			C\left(\s_{k+1;1}- \s_{k+1}\right).
			\]
			We find that
			\[
			C |\s_{k+1;1}| \leqslant C\tl_2\cdots \tl_{k-1} \tl_k^3 \leqslant C\de \tl_1 \cdots \tl_{k-1} \tl_k^2 \leqslant \frac{1}{4} \tl_1 \cdots \tl_{k-1} \tl_{n}^2
			\]
			for sufficiently small $\delta$ depending only on $n,\ k$, where we use Claim 1.
			By combining the above two inequalities and applying Claim 1, we obtain the first inequality in Claim 2.
			By \eqref{z-3}, we obtain the second inequality in Claim 2.

			\textit{\textbf{Claim 3}:} There exist positive constants $c_3$, $c_4$, $c_5$ and $c_6$ only depending on $n$ and $k$, such that
			\begin{equation}
				\label{claim-3}
				c_5\s_k (\tl) \leqslant c_3 \tl_2 \cdots \tl_k^2 \leqslant \s_{k-1;1} \leqslant c_4 \tl_2 \cdots \tl_k^2\leqslant c_6\de^2\s_{k-1} (\tl).
			\end{equation}
			Note that
			\[
			- \s_{k+1} - C \tl_2 \cdots \tl_k^3 \leqslant - \tl_1 \s_{k;1} = - \s_{k+1} + \s_{k+1;1}  
			\leqslant - \s_{k+1} + C \tl_2 \cdots \tl_k^3.
			\]
			By Claim 1 and Claim 2, we see that
			\begin{equation}
				\label{sk+1}
				- \s_{k;1} \tl_1  
				\leqslant - \s_{k+1} + C \delta \tl_1 \tl_2 \cdots \tl_k^2
				\leqslant - (1 + C \delta) \s_{k+1} .
			\end{equation}
			Now by \eqref{sit-2} and the above inequality \eqref{sk+1}, it follows that
			\[
			\tl_1 \s_{k-1;1} = \s_k - \s_{k;1} \leqslant - 16k \s_{k;1} - \s_{k;1} \leqslant \frac{-C\s_{k+1}}{\tl_1} 
			\leqslant c_4 \tl_2 \cdots \tl_k^2.
			\]
			On the other hand, by \eqref{tl k} and Claim 2, we derive that
			\[
			\tl_1 \s_{k-1;1} > - \s_{k;1}
			\geqslant \frac{- \s_{k+1} - C \tl_2 \cdots \tl_k^3}{\tl_1}
			\geqslant c_3 \tl_2 \cdots \tl_k^2.
			\]
			
			We now continue with the proof.
			Let $e_1 = (1, 0, \cdots, 0)$. Then, we can find $\zeta = (0, \zeta_2, \cdots, \zeta_{n})$ and $a \in \mathbb{R}$
			such that
			\[
			\xi =(1+a)\xi_1 \tl- a \xi_1 e_1 + \zeta
			\]
			and
			\begin{equation}
				\label{perp}
				(\zeta_k, \cdots, \zeta_{n}) \perp (\tl_k, \cdots, \tl_{n}).
			\end{equation}
			Set
			\[
			II_{1p}' = \frac{\s_{k-2;1p}}{\s_k} - 
			\frac{\s_{k-1;p}}{\s_k} \frac{\s_{k-2;1}}{\s_{k-1}} 
			- \frac{\s_{k-1;1}}{\s_k} \frac{\s_{k-2;p}}{\s_{k-1}}
			+ \frac{\s_{k-2;1}}{\s_{k-1}} \frac{\s_{k-2;p}}{\s_{k-1}}
			\]
			and
			\[
			II_{1p}'' = \frac{\s_{k-2;1}}{\s_{k-1}} \frac{\s_{k-2;p}}{\s_{k-1}}
			- \frac{\s_{k-3;1p}}{\s_{k-1}}.
			\]
			By Lemma \ref{qk,con}, we have 
			\begin{equation}
				\label{q''}
				\begin{aligned}
					-\frac{\partial_\xi^2 q_k}{q_k} 
					= &\; - (-a\xi_1 e_1 + \zeta)_p \frac{q_k^{pp,qq}}{q_k} (- a\xi_1 e_1 + \zeta)_q\\
					= &\; - 2a^2 \xi_1^2 \frac{\s_{k-2;1}}{\s_{k-1}} 
					\Big(\frac{\s_{k-2;1}}{\s_{k-1}} - \frac{\s_{k-1;1}}{\s_k}\Big) 
					- \frac{\partial_\zeta^2 q_k}{q_k} + 2a \sum_{p>1} \xi_1 \zeta_p II_{1p}.
				\end{aligned}
			\end{equation}
			where
			\[\begin{aligned}
				II_{1p} 
				= &\; II_{1p}' + II_{1p}''.
			\end{aligned}\]
			
			Then we estimate $II_{1p}''$. A direct computation gives
			\[
			II_{1p}''
			= \frac{\s_{k-2;1p}^2-\s_{k-1;1p} \s_{k-3;1p}  }{\s_{k-1}^2}.
			\]
			Hence, by Claim 3, we have
			\begin{equation}
				\label{II-1}
				\Big|II_{1i}'' \Big| \leqslant 
				\frac{C \tl_k^2}{\tl_i^2}\leqslant C\frac{\s_{k-1;1}\tl_k^2}{\s_k\tl_i^2}
			\end{equation}
			for $2\leqslant i\leqslant k-1$.
			Note that, for $2\leqslant i\leqslant k-1$, \eqref{II'} reduces to
    \begin{equation}\label{II'}
            \begin{aligned}
				II_{1i}'=&
\frac{1}{\tl_1\tl_i}\left( \frac{\s_{k;1i}}{\s_k} -
			\frac{\s_{k;i}}{\s_k} \frac{\s_{k-1;1}}{\s_{k-1}} 
			- \frac{\s_{k;1}}{\s_k} \frac{\s_{k-1;i}}{\s_{k-1}}
			+ \frac{\s_{k-1;1}}{\s_{k-1}} \frac{\s_{k-1;i}}{\s_{k-1}}\right)\\
                = &\frac{\s_{k;1i}}{\tl_i \s_k} - \frac{\s_{k-1;i} \s_{k;1}}{ \tl_i \s_{k-1} \s_k} + \frac{\s_{k-1;1}}{\s_k}\Big(\frac{\s_{k-1;i}}{\s_{k-1} } - \frac{\s_k\s_{k-2;i}}{\s_{k-1}^2}\Big).
			\end{aligned}
            \end{equation}
       See also (3.52) of \cite{Zhang} for reference.
			A direct calculation shows that
			\[
			\frac{\s_{k;1i}}{ \tl_i \s_k} - \frac{\s_{k-1;i} \s_{k;1}}{ \tl_i \s_{k-1} \s_k} = \frac{1}{ \s_k} \frac{\s_{k;1i} \s_{k-2;i} - \s_{k-1;i} \s_{k-1;1i}}{\s_{k-1}}
			\]
			and
			\[
			\s_{k;1i} \s_{k-2;i} - \s_{k-1;i} \s_{k-1;1i}
			\leqslant C \Big( \frac{\tl_1 \cdots \tl_k^3}{\tl_1 \tl_i} \frac{\tl_1 \cdots \tl_{k-1}}{\tl_i}  + \frac{\tl_1 \cdots \tl_k}{\tl_i} \frac{\tl_1 \cdots \tl_{k}^2}{\tl_1 \tl_i} \Big).
			\]
			It follows that
			\begin{equation}
				\label{II-2}
				\begin{aligned}
					\Big| \frac{\s_{k;1i}}{ \tl_i \s_k} - \frac{\s_{k-1;i} \s_{k;1}}{ \tl_i \s_{k-1} \s_k} \Big|
					\leqslant C \frac{ \tl_1\cdots \tl_k^3}{\s_k\tl_i^2}
					\leqslant C\frac{\s_{k-1;1} \tl_k}{\s_k\tl_i^2},
				\end{aligned}
			\end{equation}
			where we used Claim 3 in the last inequality.
			Also, by Claim 3, we have
			\begin{equation}
				\label{II-3}
				\frac{\s_{k-1;1}}{\s_k}\Big| \frac{\s_{k-1;i}}{\s_{k-1}} - \frac{ \s_k\s_{k-2;i}}{\s_{k-1}^2} \Big|
				\leqslant C\frac{\s_{k-1;1}}{\s_k}\left(\frac{\tl_1 \cdots \tl_k}{\tl_i \s_{k-1} } + 
				\frac{(\tl_1 \cdots \tl_{k-1}\tl_k)^2}{\tl_i \s_{k-1}^2}\right)
				\leqslant C\frac{\s_{k-1;1}\tl_k}{\s_k\tl_i}.
			\end{equation}
			
			Combining \eqref{II-1}, \eqref{II-2} and \eqref{II-3}, we therefore obtain that, for $2\leqslant i \leqslant k-1$,
			\begin{equation}
				\label{2a}
				\begin{aligned}
					2a\xi_1 \zeta_i II_{1i}
					\geqslant &\; - 2C |a \xi_1 \zeta_i|\frac{ \s_{k-1;1} }{\s_k}
					\frac{\tl_k}{\tl_i^2} \\
					\geqslant &\; - \varepsilon \frac{\s_{k-1;1} }{\s_k\tl_i^3} \zeta_i^2 
					- \frac{C a^2 \tl_k^2}{\varepsilon} \frac{\s_{k-1;1} \xi_1^2}{\s_k\tl_i }. 
				\end{aligned}
			\end{equation}
			
			Now we estimate $II_{1p}$ with $k \leqslant p \leqslant n$.  Note that, \eqref{II'} reduces to 
			\[
			\begin{aligned}
				II_{1p}'
                =&\;-\frac{\s_{k-1;1p}}{\tl_1\s_k} +
			\frac{\s_{k-1;p}}{\s_k} \frac{\s_{k-1;1}}{\tl_1\s_{k-1}} 
			- \frac{\s_{k-1;1}}{\s_k} \frac{\s_{k-2;p}}{\s_{k-1}}
			+ \frac{\s_{k-2;1}}{\s_{k-1}} \frac{\s_{k-2;p}}{\s_{k-1}}\\
                  = &\; \frac{\s_{k;1p} - \s_k}{\tl_1^2\s_k} + \frac{\s_{k-1} \tl_p - \s_{k-2;p} \tl_p^2}{\tl_1^2\s_k}\\
				&\; + \frac{\s_{k-1;p}}{\s_k} \frac{\s_{k-1;1}}{ \s_{k-1}} + \frac{\s_{k-2;p}}{\s_{k-1}} \Big(\frac{\s_{k-2;1}}{\s_{k-1}} - \frac{\s_{k-1;1}}{\s_k}\Big).
			\end{aligned}
			\]
          See also (3.58) of \cite{Zhang} for reference.
			By Claim 3 and Lemma \ref{sigmak} (3), we note that
			\[
			\Big| \frac{\s_{k;1p} - \s_k}{\s_k} \Big| \leqslant C \frac{\tl_2 \cdots \tl_k^2}{\s_k}+ 1\leqslant C \frac{\s_{k-1;1}}{\s_k}.
			\]
			A direct computation shows that
			\[
			\frac{\s_{k-2;p} \tl_p^2}{ \s_k} 
			\leqslant 
			\frac{C \tl_1 \cdots \tl_{k-2} \tl_p^2 }{ \s_k} \leqslant C\frac{\s_{k-1;1}}{\s_k\tl_{k-1}}
			\]
			by Claim 1 and Claim 3.
			We further compute that
			\[
			\frac{\s_{k-1;p}}{\s_k} \frac{\s_{k-1;1}}{\s_{k-1}} 
			\leqslant C\frac{ \tl_1 \cdots \tl_{k-1} \s_{k-1;1}}{\s_{k-1}\s_k} 
			\leqslant C \frac{\s_{k-1;1}}{\s_k}
			\]
			and
			\[
			\Big| \frac{\s_{k-2;p}}{\s_{k-1}} \Big(\frac{\s_{k-2;1}}{\s_{k-1}} - \frac{\s_{k-1;1}}{\s_k} \Big) \Big|
			\leqslant \frac{C}{\tl_{k-1}} 
            \Big(1 + \frac{\s_{k-1;1}}{\s_k}\Big) 
			\leqslant \frac{C \s_{k-1;1}}{\tl_{k-1} \s_k}.
			\]
			Note that for $p\geqslant k$,
			\[
			\Big|II_{1p}'' \Big| \leqslant 
			C\leqslant C\frac{\s_{k-1;1}}{\s_k}.
			\]
			By \eqref{perp} and the above five inequalities, we obtain, 
			for $k \leqslant p \leqslant n$, that
			\begin{equation}
				\label{2a'}
				\begin{aligned}
					2a \xi_1 \zeta_p II_{1p}
					\geqslant &\; - 2C |a \xi_1 \zeta_p| \frac{\s_{k-1;1}}{\s_k} \frac{1}{\tl_{k-1}} \\
					\geqslant &\; - \varepsilon \frac{\s_{k-1;1}  }{\s_k\tl_{k-1} \tl_k^2} \zeta_p^2
					- \frac{C a^2  \tl_k^2}{\varepsilon} \frac{\s_{k-1;1} \xi_1^2}{\s_k\tl_{k-1} }.
				\end{aligned}
			\end{equation}
			
			For the first term of the last equality in \eqref{q''}, it follows from computation that
			\[
			- 2a^2 \xi_1^2 \frac{\s_{k-2;1}}{\s_{k-1}} 
			\Big(\frac{\s_{k-2;1}}{\s_{k-1}} - \frac{\s_{k-1;1}}{\s_k}\Big)
			= 2a^2 \xi_1^2 \frac{\s_{k-1} - \s_{k-1;1}}{ \s_{k-1}} 
			\Big(\frac{- \s_{k;1}}{\s_k} + \frac{ \s_{k-1;1}}{\s_{k-1}}\Big).
			\]
			Hence, by Claim 3 and \eqref{tl k}, we have
			\[\begin{aligned}
				- 2a^2 \xi_1^2 \frac{\s_{k-2;1}}{\s_{k-1}} 
				\Big(\frac{\s_{k-2;1}}{\s_{k-1}} - \frac{\s_{k-1;1}}{\s_k}\Big) 
				\geqslant - 2 a^2 \xi_1^2\Big(1 - C\delta \Big)
				\frac{ \s_{k;1}}{ \s_k}.
			\end{aligned}\]
			Substituting \eqref{2a}, \eqref{2a'} and the preceding inequality into \eqref{q''}, we obtain:
			\begin{equation}
				\label{q''-1}
				\begin{aligned}
					-\frac{\partial_\xi^2 q_k}{q_k} 
					\geqslant &\; - \frac{\partial_\zeta^2 q_k}{q_k} - 2 a^2 \xi_1^2\Big(1 - C\delta \Big)
					\frac{ \s_{k;1}}{\s_k} - \varepsilon \sum_{k \leqslant p \leqslant n} \frac{\s_{k-1;1}  }{\s_k\tl_{k-1} \tl_k^2} \zeta_p^2\\
					&\; - \frac{C a^2 \tl_k^2 }{\varepsilon} \frac{\s_{k-1;1} \xi_1^2 }{\s_k \tl_{k-1}} - \varepsilon \sum_{2 \leqslant i < k} \frac{\s_{k-1;1} }{\s_k\tl_i^3} \zeta_i^2 
					- \frac{C a^2 \tl_k}{\varepsilon} \frac{\s_{k-1;1} \xi_1^2}{ \s_k}.
				\end{aligned}
			\end{equation}
			By Lemma \ref{decom} (see Lemma 2.8 in \cite{Zhang} for the proof) and Claim 3, there exists a positive constant $b$ depending on $n$ and $k$ such that 
			\[\begin{aligned}
				-\frac{\partial_\zeta^2 q_k}{q_k} 
				\geqslant &\; b \frac{\tl_1 \cdots \tl_{k-2}}{\s_k} \Big( \sum_{2 \leqslant i < k}  \frac{\tl_{k-1} \tl_k^2}{\tl_i^3} \zeta_i^2 + \sum_{k \leqslant p \leqslant n} \zeta_p^2 \Big)\\
				\geqslant &\; b \sum_{2 \leqslant i < k} \frac{\s_{k-1;1}}{\s_k\tl_i^3} \zeta_i^2 
				+ b \sum_{k \leqslant p \leqslant n} \frac{\s_{k-1;1}}{\s_k\tl_{k-1} \tl_k^2}\zeta_p^2,
			\end{aligned}\] 
			where we used Claim 3 in the last inequality.
			Let $ \varepsilon = b/2$. Then, it follows from \eqref{q''-1} that
			\begin{equation}
				\label{q''-2}
				\begin{aligned}
					-\frac{\partial_\xi^2 q_k}{q_k} 
					\geqslant &\; \frac{b}{2} \sum_{2 \leqslant i < k} \frac{\s_{k-1;1} }{\s_k\tl_i^3} \zeta_i^2 + \frac{b}{2} \sum_{k \leqslant p \leqslant n} \frac{\s_{k-1;1} }{\s_k\tl_{k-1} \tl_k^2}\zeta_p^2 \\
					&\; -2 a^2 \xi_1^2 \Big(1 - C\delta \Big)
					\frac{\s_{k;1}}{ \s_k} - \frac{C a^2 \de}{b} \frac{\s_{k-1;1} \xi_1^2}{\s_k},
				\end{aligned}
			\end{equation}
			where $b > 0$ only depends on $n$ and $k$.
			By \eqref{main-4} and the above inequality,
			we arrive at 
			\begin{equation}
				\label{main-6}
				\begin{aligned}
					Q \geqslant &\;\frac{b}{2} \sum_{2 \leqslant p < k} \frac{\s_{k-1;1} }{\s_k\tl_p^3} \zeta_p^2 + \frac{b}{2} \sum_{k \leqslant p \leqslant n} \frac{\s_{k-1;1} }{\s_k\tl_{k-1} \tl_k^2}\zeta_p^2 - 2 a^2 \Big(1 - C\delta \Big) \frac{ \s_{k;1}}{ \s_k}\xi_1^2\\
					&\; - \frac{C a^2 \de}{b} \frac{\s_{k-1;1} }{\s_k}\xi_1^2
					+ \Big( 1 + \frac{1}{8k} \Big)  \xi_1^2 + \Big( 2 - C\delta^{\frac{1}{k-1}} \Big)
					\sum_{i > 1} \frac{\s_k^{ii} \xi_i^2}{ \s_k}.
				\end{aligned}
			\end{equation}

			We now address the last term in the above inequality. 
			
			\textbf{Notation}: there exists a $1\leqslant\ell\leqslant k-1 $, such that $\tl_\ell \geqslant\sqrt{\de}$ and $\tl_{\ell+1} < \sqrt{\de} $. 
		Then
			\begin{equation}
				\label{i>1}
				\begin{aligned}
					\sum_{i > \ell} \frac{\s_k^{ii} \xi_i^2}{ \s_k}
					= &\; \sum_{i > \ell} \frac{\s_k^{ii} }{ \s_k} \Big((1+a) \xi_1 \tl_i + \zeta_i \Big)^2\\
					\geqslant &\; (1+a)^2 \xi_1^2 \sum_{i > \ell} \frac{\s_k^{ii} \tl_i^2}{\s_k} - 2\Big|(1+a) \xi_1 \sum_{i > \ell} \frac{\s_k^{ii} \tl_i \zeta_i}{ \s_k}\Big|\\
					\geqslant &\; (1+a)^2 \xi_1^2 \sum_{i > \ell} \frac{\s_k^{ii} \tl_i^2}{ \s_k} - 2\Big|(1+a) \xi_1 \sum_{k \leqslant p\leqslant n} \frac{\s_{k-2;p} \tl_p^2 }{ \s_k} \zeta_p \Big|\\
					&\; - 2\Big|(1+a) \xi_1 \sum_{\ell < i\leqslant k-1} \frac{\s_{k-1;i} \tl_i}{ \s_k} \zeta_i \Big|,
				\end{aligned}
			\end{equation}
			where in the second inequality we use \eqref{perp} and obtain
			\[
			\sum_{k\leqslant p \leqslant n} \s_{k-1;p} \tl_p \zeta_p =
			\sum_{k\leqslant p \leqslant n} (\s_{k-1} \tl_p \zeta_p - \s_{k-2;p} \tl_p^2 \zeta_p)
			= - \sum_{k\leqslant p \leqslant n} \s_{k-2;p} \tl_p^2 \zeta_p.
			\]
			By Lemma \ref{sigmak} (3), we obtain that
			\begin{equation}
				\label{z-1'}
				\begin{aligned}
					\sum_{\ell + 1 \leqslant j \leqslant n} \s_{k-1;j} \tl_j^2 
					= \sum_{\ell + 1 \leqslant i \leqslant n} \tl_i \s_k - \sum_{1 \leqslant i \leqslant \ell} \s_{k+1;i}-(k+1-\ell)\s_{k+1}.
				\end{aligned}
			\end{equation}
			We find that
			\begin{equation}
				\label{s_k+1}
				\sum_{1 \leqslant i \leqslant \ell} |\s_{k+1;i}| \leqslant \frac{C}{\tl_{\ell}} \tl_1 \cdots \tl_{k-1} \tl_k^3 \leqslant C \sqrt{\de}\tl_1 \cdots \tl_{k-1} \tl_k^2.
			\end{equation}
			Observe that, inserting Claim 3 and \eqref{sk+1} into \eqref{z-1'}, we obtain
			\begin{equation}
				\label{i>ell}
				\begin{aligned}
					\sum_{i > \ell} \frac{\s_k^{ii} \tl_i^2}{ \s_k} 
					\geqslant &\; - \frac{ (k+1-\ell) \s_{k+1} }{\s_k} 
					- C \sqrt{\de}\frac{\tl_1 \cdots \tl_k^2}{ \s_k}\\
					\geqslant &\; - \frac{k+1-\ell}{1 + C \delta} \frac{ \s_{k;1} }{\s_k}
					- C\sqrt{\de} \frac{\s_{k-1;1}}{ \s_k}.
				\end{aligned}
			\end{equation}
			And for $k \leqslant p \leqslant n$, by Claim 3 we see that
			\[
			\s_{k-2;p} \tl_p^2 \leqslant C \tl_1 \cdots \tl_{k-2} \tl_k^2 \leqslant C\frac{\s_{k-1; 1} }{\tl_{k-1}}.
			\]
			By the above inequality and the Cauchy–Schwarz inequality, we obtain the following estimate: 
			\begin{equation}
				\label{p>=k}
				\begin{aligned}
					2\Big|(1+a) \xi_1 \frac{\s_{k-2;p} \tl_p^2}{ \s_k} \zeta_p \Big|
					\leqslant &\; \frac{b}{4} \frac{\s_{k-1;1} }{\s_k\tl_{k-1} \tl_k^2} \zeta_p^2 + \frac{C(1+a)^2}{b} \frac{\s_{k-1;1} \tl_k^2}{\s_k\tl_{k-1} } \xi_1^2\\
					\leqslant &\; \frac{b}{4} \frac{\s_{k-1;1} }{\s_k\tl_{k-1} \tl_k^2} \zeta_p^2 + \frac{C\de (1+a)^2}{b}\frac{ \s_{k-1;1} }{\s_k} \xi_1^2,
				\end{aligned}
			\end{equation}
			where Claim 3 is used in the last inequality.
			For $\ell < i \leqslant k-1$, by \eqref{sit-2}, we find that
			\[
			\s_{k} \leqslant -16 k \s_{k;1} \leqslant C \tl_2 \cdots \tl_k^2  \leqslant \frac{C \s_{k-1;1} }{\tl_i}
			\]
			and, also
			\[
			- \s_{k;i} \leqslant \frac{C \tl_1 \cdots \tl_k^2}{\tl_i} \leqslant \frac{C \s_{k-1;1} }{\tl_i}.
			\]
			By the above two inequalities and $\s_{k-1;i} \tl_i = \s_k - \s_{k;i}$, we obtain
			\begin{equation}
				\label{p<k-1}
				\begin{aligned}
					2\Big|(1+a) \xi_1 \frac{\s_{k-1;i} \tl_i}{ \s_k} \zeta_i \Big|
					\leqslant &\; \frac{b}{4} \frac{\s_{k-1;1}  \zeta_i^2}{\s_k\tl_i^3}
					+ \frac{C(1+a)^2}{b}\frac{\s_{k-1;1} \tl_i}{\s_k} \xi_1^2\\
					\leqslant &\; \frac{b}{4} \frac{\s_{k-1;1} \tl_1^2 \zeta_i^2}{\s_k\tl_i^3}
					+ \frac{C \sqrt{\de} (1+a)^2 }{b}\frac{\s_{k-1;1}}{\s_k} \xi_1^2,	
				\end{aligned}
			\end{equation}
			where we used the Cauchy-Schwarz inequality and the \textbf{Notation}. 
			Substituting the inequalities \eqref{i>ell}, \eqref{p>=k} and \eqref{p<k-1} into \eqref{i>1}, we conclude that
		\begin{equation}\label{sk 1}
		\begin{aligned}
				\sum_{i > \ell} \frac{\s_k^{ii} \xi_i^2}{\s_k}
				\geqslant &\; - (1+a)^2 \Big(\frac{k+1 -\ell}{1 + C\delta} \frac{ \s_{k;1} }{ \s_k} \xi_1^2 
				+ C\sqrt{\de} \frac{\s_{k-1;1}}{\s_k} \xi_1^2 \Big)\\	
				&\; - \frac{b}{4} \sum_{k \leqslant p\leqslant n} \frac{\s_{k-1;1}}{\s_k\tl_{k-1} \tl_k^2} \zeta_p^2 -\frac{C \de (1+a)^2 }{b}\frac{\s_{k-1;1}}{\s_k}\xi_1^2\\
				&\; - \frac{b}{4} \sum_{\ell < i \leqslant k-1} \frac{\s_{k-1;1} }{\s_k\tl_i^3} \zeta_i^2
				- \frac{C \sqrt{\de} (1+a)^2 }{b} \frac{\s_{k-1;1}}{\s_k} \xi_1^2.
			\end{aligned}    
		\end{equation}
			Let 
            \[
            \tilde a = a^2 + 2(1+a)^2.
            \]
			Substituting \eqref{sk 1} into \eqref{main-6}, we obtain:
			\begin{equation}
				\label{main-7}
				\begin{aligned}
					Q \geqslant & - 2\tilde a  \Big(1 - C\delta^{\frac{1}{k-1}} \Big)
					\frac{\s_{k;1}}{\s_k}\xi_1^2 + \Big(1 + \frac{1}{8k}\Big) \xi_1^2 - \frac{C \tilde a \sqrt{\de}}{b} \frac{\s_{k-1;1}}{\s_k} \xi_1^2.
				\end{aligned}
			\end{equation}
			where in the inequality we used $k + 1 - \ell \geqslant 2$.
			By the proof of Claim 3, we observe that
			$ \s_{k-1;1} \leqslant - (16k + 1) \s_{k;1}$.
We have
\[\begin{aligned}
- \tilde a  \frac{ \s_{k;1}}{ \s_k} \xi_1^2
\geqslant \frac{\tilde a }{C} \frac{\s_{k-1;1}}{\s_k} \xi_1^2.
\end{aligned}\]
We obtain from \eqref{main-7} that
\begin{equation}
	\label{main-9}
		Q \geqslant 2\tilde a (1 - C \delta) \Big( \frac{\s_{k-1;1} }{ \s_k} - 1\Big) \xi_1^2 + \Big( 1 + \frac{1}{8k} \Big) \xi_1^2
\end{equation}
by choosing $\de$ small enough. Note that
$2\tilde a \geqslant 4/3 >1 + 1/8k$. 
            Then, since
			$\s_{k;1}<0$, by choosing $\delta \ll 1$ depending only on $n,\ k$,
			we obtain from \eqref{main-9} that
			\begin{equation}
				\label{main-10}
				\begin{aligned}
					Q 
					\geqslant &\; (1 - C \delta) \Big(1 + \frac{1}{8k}\Big) 
					\frac{\s_{k-1;1} \xi_1^2}{ \s_k} 
                    \geqslant (1+\gamma) \frac{\s_{k-1;1} \xi_1^2}{ \s_k}.
				\end{aligned}
			\end{equation}
		\end{proof}

Now we can prove the concavity inequality for the sum Hessian operator \eqref{sumH}.
Let
    \[
     F_i (\la)= \frac{\partial F (\la)}{\partial \la_i}\;
     \mbox{and}\;
     F_{ij} (\la) = \frac{\partial^2 F (\la)}{\partial \la_i \partial \la_j}.
    \]
We have
\begin{lemma}
\label{Fconcavity}
    Suppose $(\la,y) \in \Gamma_k^{n+m}$, $\la_1\geqslant\la_2\geqslant\cdots\geqslant \la_n \geqslant -\delta \lambda_1$ and $y_m\geq -\delta\lambda_1$. Then, for sufficiently small $\delta \in (0, 1)$, sufficiently large $K_0$ depending only on $n$, $k$ and $K_1=\delta^{-1}y_{\min}$  such that $\lambda_1>\max\lbrace K_0F^{\frac{1}{k}},-K_1\rbrace$, the following inequality holds at $\lambda^{\downarrow}$:
    \begin{align}\label{key F}
  -\sum_{\substack{p \neq q}} \frac{F_{pq} \xi_p \xi_q}{F} + 2\frac{\left(\sum_{i=1}^{n} F_{i} \xi_i\right)^2}{F^2} + 2\sum_{i>1} \frac{F_{i}\xi_i^2}{(1+2\delta)\lambda_1 F} \geqslant (1+\gamma)\frac{F_{1} \xi_1^2}{\lambda_1 F}   
    \end{align}
    for any vector $\xi = (\xi_1, \ldots, \xi_n) \in \mathbb{R}^{n}$ with $\gamma = 1/(1+16k)$.
\end{lemma}
    \begin{proof}  
Define
\[\begin{aligned}
	\hat{\la} = &\;(\la_1,\ldots,\la_n,y_1,\dots,y_m),\\
    \hat{\xi}= &\; (\xi_1,\ldots,\xi_n,0,\ldots,0).
\end{aligned}\]
Hence, by Lemma \ref{key-c},
    \[F(\lambda)=\sigma_k^{(n)} (\la)+\sum_{r=1}^m a_r \sigma_{k-r}^{(n)} (\la)= \sigma_k^{(n+m)} (\hat{\la}).\]
    Since $\la_i = \hat \la_i$ for $1\leqslant i \leqslant n$, we have
    \[
    F_i(\la) = \sigma_{k-1;i}^{(n)} (\la)+\sum_{r=1}^m a_r \sigma_{k-r-1;i}^{(n)} (\la)
    = \sigma_{k-1;i}^{(n+m)} (\hat{\la}).
    \]
    Similarly, we see that
    \[
    F_{ij} (\la) = \sigma_{k-2;ij}^{(n+m)} (\hat{\la})\; \mbox{for}\; 1 \leqslant i\neq j \leqslant n.
    \]
    Then we rearrange $\hat{\lambda}$ and $\hat{\xi}$ accordingly:
\begin{align*}
\bar{\lambda}=\hat{\lambda}^{\downarrow}
= &\; (\lambda_1,\hat \lambda_{\tau(2)},\ldots, \hat \lambda_{\tau(n+m)}),\\
\bar{\xi}= &\; (\xi_1, \hat \xi_{\tau(2)},\ldots, \hat \xi_{\tau(n+m)}),
\end{align*}
where $\tau$ is a permutation.
	Moreover, we have
	\[
	\sum_{i=1}^{n+m} \sigma_{k-1;i}^{(n+m)}(\bar{\la})\bar\xi_i= \sum_{i=1}^{n+m} \sigma_{k-1;i}^{(n+m)} (\hat{\la})\hat{\xi}_i =\sum_{i=1}^{n} F_i  (\la)\xi_i,
    \]
    \[
	\sum_{i>1}^{n+m} \sigma_{k-1;i}^{(n+m)}(\bar{\la})\bar\xi_i^2 = \sum_{i>1}^{n+m} \sigma_{k-1;i}^{(n+m)}(\hat{\la})\hat{\xi}_i^2 = \sum_{i>1}^{n} F_i (\la)\xi_i^2,
	\]
	and
	\begin{align*}
		\sum_{p\ne q}^{n+m}\sigma_{k-2;pq}^{(n+m)}(\bar{\la})\bar\xi_p\bar\xi_q
	= \sum_{p\ne q}^{n+m} \sigma_{k-2;pq}^{(n+m)}(\hat{\la})\hat{\xi}_p\hat{\xi}_q = \sum_{p\ne q}^{n}F_{pq}(\la)\xi_p\xi_q.    
	\end{align*}
By applying Lemma \ref{key} to $\s_k^{(n+m)}(\bar \la)$,	
we obtain inequality \eqref{key F}.
\end{proof}
For the same reason, it follows that $F^{\frac{1}{k}}$ is concave in the $\Gamma_k^{(n+m)}$ cone.

\section{Proof of Theorem \ref{thm1} for Case (B)}

\begin{proof}
We consider the following function for $x \in \Omega$, $\xi \in \mathbb{S}^{n-1}$
\begin{eqnarray*}
\widetilde{P}(x, \xi)= \alpha \log (-u)+\log \max \{u_{\xi\xi}(x), 1\}+\frac{L}{2}|Du|^2,
\end{eqnarray*}
where $L\gg 1$, $\alpha >0$ are constants to be determined later.
Clearly, $\widetilde{P}$ attains its maximum at some interior point $x_0 \in \Omega$ and for some $\xi(x_0) \in \mathbb{S}^{n-1}$.
Choose a normal coordinate frame $e_1,\ldots,e_n$ at $x_0$ such that $\xi(x_0)=e_1$ and $D^2 u(x_0)$ is diagonal.
Denote the eigenvalues of $D^2 u(x_0)$ by
$$\lambda_1(x_0)\geqslant\lambda_2(x_0)\geqslant \cdots \geqslant \lambda_n(x_0).$$ 
Let $\gamma(s)$ be any geodesic starting at $x_0$ satisfying $\gamma'(0) = e_i(x_0)$. We construct an
orthonormal frame $e_1(s), e_2(s), \dots, e_n(s)$ near $x_0$ extending $e_1(x_0), e_2(x_0), \dots , e_n(x_0)$
parallel to $\gamma(s)$. We construct a unit vector field $v(x)$ near $x_0$ such that $v(x_0)=e_1(x_0)$ and $$\nabla_iv(x_0) = \sum_{p\neq 1}\frac{u_{1pi}}{\la_1+1-\la_p}e_p(x_0).$$
We may also assume that $\lambda_1(x_0)> 1$ is sufficiently large.
Then we consider the function
\begin{eqnarray*}
P(x)=\alpha\log (-u)+\log\left( \nabla^2 u\left( v(x),v(x)\right) +  \langle v(x), e_1(x)\rangle^2 - 1 \right)+\frac{L}{2}|Du|^2.
\end{eqnarray*}
Note that $x_0$ is also a maximum point of $P$. At $x_0$, we have 
\begin{align}\label{cri}
    0=\nabla_iP(x_0)=\alpha\frac{u_i}{u}+\frac{u_{11i}}{u_{11}}+Lu_i\lambda_i
\end{align}
   and 
   \begin{equation}\label{hess 31}
   \begin{aligned}
    0\geqslant &  \nabla_{ii}P(x_0) \\
    =&\alpha\frac{u_{ii}}{u}-\alpha\frac{u_i^2}{u^2}+\frac{u_{11ii}}{u_{11}}+\sum_{l>1}\frac{2u_{1li}^2}{u_{11}(u_{11}-u_{ll}+1)}\\
    &-\frac{u_{11i}^2}{u_{11}^2} + L \sum_l u_{li}^2 + L \sum_l u_{lii}u_l.
   \end{aligned}    
   \end{equation}

Define by
\[
F^{ij} = \frac{\partial F}{\partial u_{ij}}
\;\mbox{and}\;
F^{ij,kl} = \frac{\partial^2 F}{\partial u_{ij} \partial u_{kl}}.
\]
Since at $x_0$, $\{u_{ij}(x_0)\}$ is diagonal, we have
\[
F^{ij} = \delta_{ij} F_i.
\]   
By $F(D^2 u)=\psi(x,u, D u)$, we have at $x_0$ 
\begin{equation}\label{gra}
   \nabla_l F=\psi_l+\psi_uu_l+\psi_{u_l} \la_{l}
   = \nabla_l \psi
\end{equation}
and 
\begin{equation}\label{hess s}
\nabla^2_{11} F=F^{ij}u_{ij11} + F^{pq,rs}u_{pq1}u_{rs1} = \nabla^2_{11} \psi.
\end{equation}
By a direct calculation, there exists some constant $C>0$ depending on $\|\psi\|_{C^2}$ and $\|u\|_{C^1}$ such that 
\begin{align}
\label{psi}
    |\nabla_{l} \psi|\leqslant Cu_{11},\quad \nabla^2_{11} \psi \geqslant -Cu_{11}^2+\sum_l \psi_{u_l}u_{11l}
\end{align}
for $u_{11}$ sufficiently large.
As $(\lambda(D^2 u), y)\in \Gamma_k^{n+m}$, we have 
\begin{align*}
    \s^{(n+m)}_{k-1;y_i}>0.
\end{align*}
Therefore, by Condition \eqref{con1}, we obtain
\begin{align}\label{kF}
F^{ij}u_{ij}=k\s_k^{(n+m)}-\sum_{i=1}^m y_i \s^{(n+m)}_{k-1;y_i} \leqslant kF.
\end{align}
By \eqref{cri}, we derive that
\begin{align}
    \label{psi_u}
    \frac{\psi_{u_l}u_{11l}}{u_{11}} + \psi_{u_l} L \lambda_l u_l = -\alpha\frac{\psi_{u_l}u_l}{u}.
\end{align}
Contracting \eqref{hess 31} with $F^{ij}$, and by \eqref{gra}, \eqref{hess s}, \eqref{psi}, \eqref{kF} and \eqref{psi_u}, we obtain 
\begin{equation}\label{hess 32}
\begin{aligned}
    0 \geqslant &\; \alpha\frac{kF}{u} - \alpha\frac{F^{ii}u_i^2}{u^2}-\frac{F^{pq,rs}u_{pq1}u_{rs1}}{u_{11}} 
- \alpha\frac{\psi_{u_l}u_l}{u}\\
    &+2\sum_{p>1}\frac{F^{ii}u_{1pi}^2}{(1+ 2\delta) u_{11}^2} - \frac{F^{ii}u_{11i}^2}{u_{11}^2} + LF^{ii}u_{ii}^2 -Cu_{11} -CL.
\end{aligned}
\end{equation}
Recall the following formula in \cite{And07},
\begin{equation}\label{fppqq1}
 \begin{aligned}
-F^{pq,rs}u_{pq1}u_{rs1}=&-F^{pp,qq}u_{pp1}u_{qq1} - 2\sum_{p>q}F^{pq,qp}u_{pq1}^2\\
\geqslant& -\sum_{p\neq q} F_{pq} u_{pp1}u_{qq1}+2\sum_{p>1}\frac{F_{p}-F_{1}}{\lambda_1-\lambda_p}u_{11p}^2,
\end{aligned}   
\end{equation}
where we used Lemma \ref{sigmak} (6) in the inequality.
Note that
\begin{equation}\label{fppqq2}
 2 \sum_{p>1} \frac{F^{ii}u_{1pi}^2}{(1+2 \delta) u_{11}^2 }\geqslant  2\sum_{p>1}\frac{F^{11}u_{11p}^2}{(1+ 2\delta) u_{11}^2} + 2\sum_{i>1}\frac{F^{ii}u_{1ii}^2}{(1+2 \delta) u_{11}^2}.
\end{equation}
Using the critical equation \eqref{cri} and the Cauchy–Schwarz inequality, we obtain for any $i\geqslant 2$ that 
\begin{align}\label{alpha}
   -\alpha\frac{F^{ii}u_i^2}{u^2}=-\frac{F^{ii}}{\alpha}\left(\frac{u_{11i}}{u_{11}}+Lu_iu_{ii}\right)^2\geqslant -\frac{2}{\alpha}\frac{F^{ii}u_{11i}^2}{u_{11}^2}-2\frac{L^2}{\alpha}F^{ii}u_i^2u_{ii}^2.
\end{align}
Then, by \eqref{fppqq1}, \eqref{fppqq2} and choosing $\alpha = L^2 \gg 4$, we obtain
\begin{equation}\label{fppqq3}
    \begin{aligned}
&-\frac{F^{pq,rs}u_{pq1}u_{rs1}}{u_{11}}+2\sum_{p>1}\frac{F^{ii}u_{1pi}^2}{(1+ 2\delta) u_{11}^2}-\left( 1+\frac{2}{\alpha}\right)\sum_{i\geqslant2}\frac{F^{ii}u_{11i}^2}{u_{11}^2}\\
\geqslant&-\frac{F^{pp,qq}u_{pp1}u_{qq1}}{u_{11}}+\left(\frac{2}{1+ 2\delta} - \frac{3}{2} \right) \sum_{p>1}\frac{F^{pp} u_{11p}^2}{ u_{11}^2} + 2\sum_{i>1}\frac{F^{ii}u_{1ii}^2}{(1+ 2\delta) u_{11}^2}\\
\geqslant &-\frac{F^{pp,qq}u_{pp1}u_{qq1}}{u_{11}}+ 2\sum_{i>1}\frac{F^{ii}u_{1ii}^2}{(1+2\delta)u_{11}^2}
\end{aligned}
\end{equation}
by assuming $\delta\leqslant \frac{1}{6}$.
By using Lemma \ref{sigmak} (4), we have 
\begin{align}
F^{11}u_{11}^2= \s_{k-1;1}^{(n+m)}(\bar{\la})u_{11}^2\geqslant \frac{k}{n+m}F u_{11}.
\end{align}
Due to Lemma \ref{Fconcavity}, combining 
\eqref{alpha} and \eqref{fppqq3}, \eqref{hess 32} becomes 
\begin{equation}\label{hess 33}
\begin{aligned}
    0\geqslant &-L^2 \frac{F^{11}u_1^2}{u^2}-2F^{ii}u_i^2 u_{ii}^2 
     + L F^{ii}u_{ii}^2 + \frac{C_3L^2}{u} - C_4L -C_5 u_{11}\\
   \geqslant &\left(\frac{Lu_{11}^2}{5}-\frac{C_1L^2}{u^2}\right)F^{11} + L\left(\frac{2kF}{5(n+m)}u_{11} + C_3\frac{L}{u} - C_4\right)\\
   & + \left(\frac{L}{5}-C_2\right) F^{ii}u_{ii}^2 +\left(\frac{LkF}{5(n+m)}-C_5\right) u_{11}.
\end{aligned}
\end{equation}
Here $C_i$, $ i=1,2,\dots,5$ all depend on $n$, $k$, $\|\psi\|_{C^2}$ and $\|u\|_{C^1}$. We choose $$L=\max\left\{ 4,5C_2,\frac{5 (n+m) C_5}{k\psi_0}\right\},$$ 
so that the last two terms in the last inequality of \eqref{hess 33} are positive. Then we obtain 
$$\max (-u)^{\alpha}\lambda_1 \leqslant \max\left\{ \sqrt{5C_1L}\sup |u|^{\alpha-1},\frac{5(n+m)LC_3\sup |u|^{\alpha-1}}{k \psi_0}, \frac{5(n+m)C_4\sup |u|^{\alpha}}{k \psi_0}\right\}$$ and complete the proof.
\end{proof}

\section{Proof of Theorem \ref{thm2} for Case (B)}

\begin{proof}
	Consider the following test function
	\[
	P = \beta \log (-u) + \log \max \{u_{\xi\xi}(x), 1\} + \frac{1}{2} |x|^2.
	\]
	Suppose that $P$ attains its maximum at the point $x_0 \in \Omega$.
   As before, we construct a unit vector field $v(x)$ near $x_0$ such that $v(x_0)=e_1(x_0)$ and $$\nabla_iv(x_0) = \sum_{p\neq 1}\frac{u_{1pi}}{\la_1+1-\la_p}e_p(x_0).$$
We may also assume that $\lambda_1(x_0)> 1$ is sufficiently large.
Then we consider the function
\begin{eqnarray*}
\tilde{P}(x)=\beta\log (-u)+\log\left( \nabla^2 u\left( v(x), v(x)\right)+ \langle v(x),e_1(x)\rangle^2 -1\right)+\frac{|x|^2}{2}.
\end{eqnarray*}
Differentiating $\tilde P$ at $x_0$, we obtain
\begin{equation}
 \label{diff-1}
0 = \frac{u_{11i}}{\lambda_1} + \frac{\beta u_i}{u} + x_i
\end{equation}
and
\begin{equation}
	\label{diff-2}
	0 \geqslant \frac{u_{11ii}}{\lambda_1} + \sum_{\ell >1} \frac{2 u_{1\ell i}^2}{\lambda_1^2(1+2\delta)} - \frac{u_{11i}^2}{\lambda_1^2} + \frac{\beta u_{ii}}{u} - \frac{\beta u_i^2}{u^2} +1.
\end{equation}
By \eqref{diff-1}, we have
\[
- \frac{\beta F^{ii} u_i^2}{u^2} \geqslant - \frac{2}{\beta} \frac{F^{ii} u_{11i}^2}{\lambda_1^2} - \frac{2}{\beta} F^{ii} x_i^2.
\]
Contracting \eqref{diff-2} with $F^{ii}$, by \eqref{kF}, \eqref{fppqq1}, \eqref{fppqq2}, \eqref{fppqq3} and the above inequality, we obtain 
\[
	\begin{aligned}
	0 \geqslant &\;-\frac{F^{pp,qq}u_{pp1}u_{qq1}}{u_{11}}+2\frac{(\nabla_1 F)^2}{u_{11} F}+2\sum_{i>1}\frac{F^{ii}u_{1ii}^2}{(1+2\delta) u_{11}^2}\\
	& - \left( 1 + \frac{2}{\beta} \right) \frac{F^{11}u_{111}^2}{u_{11}^2}+ \frac{k\beta \psi}{u} + \left( 1 - \frac{C}{\beta} \right) \sum_{i=1}^n F^{ii} - \frac{C}{\lambda_1},
\end{aligned}
\]
where $C$ depends on $||\psi||_{C^2(\Omega)}$ and the diameter of $\Omega$.
By \eqref{key F}, we obtain
\[
		0 \geqslant\frac{k\beta \psi}{u} + \Big( 1 - \frac{C}{\beta} \Big) \sum_{i=1}^n F^{ii}  - \frac{C}{\lambda_1},
\]
where $C$ depends on $||\psi||_{C^2}$ and the diameter of $\Omega$.

Recall the definition of $\hat \la$ and $\bar \la$ in Lemma \ref{Fconcavity}.
Since $k \leqslant n$, it is easy to see that
$\la_n = \hat \la_{\tau(i_0)}$ for some $i_0 \geqslant k$.
By Condition \eqref{con1} and Lemma \ref{sigmak} (9), we obtain
\[
F^{nn} = \sigma_{k-1; i_0} (\bar \la) \geqslant c(n,k) \sigma_{k-1} (\bar \la).
\]
Using Newton inequality, we have 
    \begin{align}\label{new k-1}
       \s_{k-1}(\hat{\la})\geqslant C(n,k) \s_1^{\frac{1}{k-1}} (\hat{\la})\s_k^{\frac{k-2}{k-1}}(\hat{\la})\geqslant C\lambda_1^{\frac{1}{k-1}}
    \end{align}
       for $C$ depending on $n$, $k$ and $\inf \psi$. 
We derive that
\begin{align}
    \label{sumFi}
    \sum F^{ii} \geqslant C \la_1^{\frac{1}{k-1}}.
\end{align}
 By Lemma \ref{Fconcavity} and \eqref{sumFi}, assuming $\beta$ large enough we arrive at
\[
		0 \geqslant \frac{k\beta \psi}{u} 
		+ \frac{1}{2}\sum_{i=1}^n F^{ii}- C\geqslant \frac{k\beta \psi}{u} 
		+ C\lambda_1^{\frac{1}{k-1}}- C.
\]
We finally obtain
\[
	-u\lambda_1^{\frac{1}{\beta}}\leqslant -u\lambda_1^{\frac{1}{k-1}}\leqslant  \frac{k\beta \psi}{C}  
\]
for $\beta$ sufficiently large depending on $n$, $k$, 
$\psi_0$, $||\psi||_{C^2}$ and the diameter of $\Omega$.
This completes the proof.
\end{proof}

 \section{Proof of Theorems \ref{thm1} and \ref{thm2} for Case (A)}

We claim that if $k=n$, then $(D^2 u)_{\min}>-A$, where $A = \sum_{i=1}^m y_i=a_1\geqslant 0$. Since $\hat \la = (\lambda^{\downarrow} (D^2 u), y)\in \Gamma_n^{n+m}$, we have 
\begin{align*}
    \s_{1;12\dots n-1}^{(n+m)} (\hat \la) =\la_n+\sum_{i=1}^m y_i>0,
\end{align*}
which proves the claim. Then, by the arguments in Sections 4 and 5, we complete the proof of case (A) of Theorems \ref{thm1} and \ref{thm2} when $k=n$.

Next we consider the case $k=n-1$. We claim that 
\begin{align}
    \label{n-1,n}
    \la_n\geqslant -\frac{1}{n-1}\la_1 - C
\end{align}
 for some $C>0$ depending only on $n$, $k$ and $A$. Assume that 
\begin{align}\label{ln assumption}
   \la_n< -\frac{1}{n-1}\la_1.
\end{align}
Since $(\lambda(D^2 u), y)\in \Gamma_{n-1}^{n+m}$, we have 
\begin{align*}
    \s_{1;12 \dots n-2}^{(n+m)} (\hat \la)=\la_{n-1}+\la_n+\sum_{i=1}^m y_i>0.
\end{align*}
Then, invoking \eqref{ln assumption}, we obtain
\begin{align}\label{ln-1}
 \la_{n-1}>-\la_n-\sum_{i=1}^m y_i>\frac{1}{n-1}\la_1-a_1 > 0
\end{align}
for sufficiently large $\la_1$.
Recall that 
\begin{align*}
    \s_{n-1}^{(n+m)} (\hat \la) =\la_n\s_{n-2;n}^{(n+m)} (\hat \la) + \s_{n-1;n}^{(n+m)} (\hat \la) > 0.
\end{align*}
It follows from $\la_n < 0$ that $\s_{n-1;n}^{(n+m)}>0$ and 
\begin{align}\label{lan}
 -\la_n < \frac{\s_{n-1;n}^{(n+m)} (\hat \la) }{\s_{n-2;n}^{(n+m)}(\hat \la) }.   
\end{align}
We can estimate as 
\begin{align*}
  \s_{n-1;n}^{(n+m)}(\hat \la)\leqslant \s_{n-1;n}^{(n)}(\la)+C\la_1^{n-2}.
\end{align*}
and 
\begin{align*}
   \s_{n-2;n}^{(n+m)}(\hat \la)=\s_{n-2;n}^{(n)}(\la)+\sum_{i=1}^{m}a_i\s_{n-i-2;n}^{(n)}\geqslant \s_{n-2;n}^{(n)}(\la).
\end{align*}
Combining \eqref{ln assumption} and \eqref{ln-1}, we have
\begin{align*}
       \s_{n-2;n}^{(n+m)}(\hat \la)>\la_1\dots\la_{n-2}\geqslant C\la_1^{n-2}.
\end{align*}
Thus, by the above three inequalities, \eqref{lan} becomes 
\begin{align*}
    -\la_n \leqslant \frac{\s_{n-1;n}^{(n)} (\la)}{\s_{n-2;n}^{(n)} (\la)} + C\leqslant \frac{\lambda_1}{n-1}+C,
\end{align*}
where we use the Newton-Maclaurin inequality.
Note that \eqref{fppqq3} still holds in this case for some $\alpha$ large only depending on $n$.\\

When $k=n-1$, to complete the proof of Case (A) of Theorem \ref{thm1} and Theorem \ref{thm2}, we also establish the following concavity inequality of $n-1$ sums-of-Hessian operator $F=\s_{n-1}^{(n)}+\sum_{r=1}^m a_r \s_{n-r-1}^{(n)}$. As for the special case $m=0$,   $F=\s_{n-1}$, the following inequality was first proved by Ren-Wang \cite{RW1}. A simple proof was recently provided by Tu \cite{Tu} by using the method of Zhang \cite{Zhang} and Lu-Tsai \cite{LT}. We hereby give a concise proof using our new concavity inequality in Lemma \ref{key}.
\begin{lemma}\label{key n-1}
		Suppose that $( \lambda, y ) \in \Gamma_{n-1}^{n+m}$, $\la_1\geqslant\la_2\geqslant\cdots\geqslant \la_n$ and $C>F>1/C$. Then, there exist constants $\gamma \in (0, 1)$ and $K_0>0$ depending only on $n,\ k,\ C$ and $(a_1,\dots,a_m)$, such that if $\lambda_1>K_0$, the following inequality holds at $\lambda^{\downarrow}$ :
		\begin{equation}
			\label{positive-n-1}
			- \sum_{p\neq q} \frac{F_{pq} \xi_p \xi_q}{F}
			+ 2\frac{(\sum_i F_i \xi_i)^2}{F^2} 
			+ 2 \sum_{i > 1} \frac{F_i \xi_i^2}{(\la_1-\la_i+1)F} 
			\geqslant \left(1+\gamma\right) \frac{F_1 \xi_1^2}{\lambda_1 F}
		\end{equation}
        for any vector $\xi = (\xi_1, \ldots, \xi_n) \in \mathbb{R}^{n}$.
	\end{lemma}
\begin{proof}
Without loss of generality, we assume that $|\xi| = 1$.
Combining the result in Lemma \ref{key} and Lemma \ref{Fconcavity}, we only need to consider the case 
\begin{align}\label{ln d}
  \la_n < -\delta \la_1.  
\end{align}
Denote that 
\[\tl_i=\frac{\la_i}{\la_1} \; \mbox{and} \; \tilde y_j = \frac{y_j}{\la_1},
\]
where $1 \leqslant i \leqslant n$ and $ 1\leqslant j \leqslant m$.
By \eqref{n-1,n}, \eqref{ln-1} and \eqref{ln d}, we see that for $\la_1$ sufficiently large only depending on $n$ and $(a_1,\dots,a_m)$,
\begin{align}\label{ln}
    1=\tl_1 \geqslant \tl_2 \geqslant \cdots \geqslant \tl_{n-1} > \frac{\de}{2} > 0 > -\de > \tl_n > - \frac{2}{n-1}.
\end{align}
Let $\tl=(\tl_1,\dots,\tl_n)$, $\hat{\la}=(\tl_1,\dots,\tl_n,\tilde y_1,\dots, \tilde y_m)$. Note that
\[ 
\s_k^{(n+m)}(\hat{\la})=\frac{\s_k^{(n+m)}}{\la_1^k}\sim O(\la_1^{-k}).
\]

As in the proof of Lemma \ref{key}, it suffices to verify that
    
\begin{equation}\label{k}
\begin{aligned}
Q := &\; - \sum_{p\neq q} \s_{k-2;pq}^{(n+m)}(\hat{\la})\xi_p \xi_q + \frac{\left(\sum_i \s_{k-1;i}^{(n+m)}(\hat{\la}) \xi_i\right)^2}{\s_k^{(n+m)}(\hat{\la})} \\
&\; + 2 \sum_{i > 1} \frac{\s_{k-1;i}^{(n+m)}(\hat{\la}) \xi_i^2}{\tl_1-\tl_i+\la_1^{-1}} - \left(1+\frac{1}{2n}\right) \frac{\s_{k-1;1}^{(n+m)}(\hat{\la}) }{\tl_1}\xi_1^2 \geqslant 0.
\end{aligned}
\end{equation}
by assuming $\gamma\leq\frac{1}{2n}.$
Since 
\[
\s_k^{(n+m)}=\tl_p\tl_q\s_{k-2;pq}^{(n+m)}+(\tl_p+\tl_q)\s_{k-1;pq}^{(n+m)}+\s_{k;pq}^{(n+m)},
\]
then we have 
\[
\s_{k-2;pq}^{(n+m)}(\hat{\la})\xi_p \xi_q=\frac{\s_k^{(n+m)}}{\tl_p\tl_q}\xi_p \xi_q-\frac{(\tl_p+\tl_q)\s_{k-1;pq}^{(n+m)}}{\tl_p\tl_q}\xi_p \xi_q-\frac{\s_{k;pq}^{(n+m)}}{\tl_p\tl_q}\xi_p \xi_q.
\]
Combining 
\[\s_{k;p}^{(n+m)}=\tl_q\s_{k-1;pq}^{(n+m)}+\s_{k;pq}^{(n+m)},\]
we obtain 
 \begin{equation}\label{k,1}
F_{pq}(\tl)\xi_p \xi_q=\frac{\s_k^{(n+m)}}{\tl_p\tl_q}\xi_p \xi_q-\frac{\s_{k;p}^{(n+m)}+\s_{k;q}^{(n+m)}}{\tl_p\tl_q}\xi_p \xi_q+\frac{\s_{k;pq}^{(n+m)}}{\tl_p\tl_q}\xi_p \xi_q.
\end{equation}
Besides, as 
\begin{align}\label{sk-1 i}
  \s_{k-1;i}^{(n+m)}=\frac{\s_k^{(n+m)}}{\tl_i}-\frac{\s_{k;i}^{(n+m)}}{\tl_i},  
\end{align}
we have 
\begin{equation}\label{k,2}
\begin{aligned}
 &\frac{1}{\s_k^{(n+m)}} \s_{k-1;p}^{(n+m)}\s_{k-1;q}^{(n+m)}\xi_p\xi_q\\
 =&\; \frac{1}{\s_k^{(n+m)}}\left(\frac{\s_k^{(n+m)}}{\tl_p}-\frac{\s_{k;p}^{(n+m)}}{\tl_p}\right)\left(\frac{\s_k^{(n+m)}}{\tl_q}-\frac{\s_{k;q}^{(n+m)}}{\tl_q}\right)\xi_p\xi_q\\
=&\; \frac{\s_k^{(n+m)}}{\tl_p\tl_q}\xi_p\xi_q - \frac{\s_{k;p}^{(n+m)}+\s_{k;q}^{(n+m)}}{\tl_p\tl_q}\xi_p\xi_q + \frac{1}{\s_k^{(n+m)}}\frac{\s_{k;p}^{(n+m)}}{\tl_p}\frac{\s_{k;q}^{(n+m)}}{\tl_q}\xi_p\xi_q.
\end{aligned}
\end{equation} 
By \eqref{k,1}, \eqref{k,2} and \eqref{sk-1 i}, we have the first two terms in \eqref{k} become
\begin{equation}
\label{First Two term}
\begin{aligned}
& - \sum_{p\neq q} F_{pq}(\tl)\xi_p \xi_q
		+  \frac{\left(\sum_i F_i(\tl) \xi_i\right)^2}{F} \\
        =&-\sum_{p\neq q} \frac{\s_{k;pq}^{(n+m)}}{\tl_p\tl_q}\xi_p \xi_q+\sum_{p\neq q}\frac{1}{\s_k^{(n+m)}}\frac{\s_{k;p}^{(n+m)}}{\tl_p}\frac{\s_{k;q}^{(n+m)}}{\tl_q}\xi_p\xi_q\\
&+\sum_{i=1}^n\frac{\s_k^{(n+m)}}{\tl_i^2}\xi_i^2-	\sum_{i=1}^n\frac{2\s_{k;i}^{(n+m)}}{\tl_i^2}\xi_i^2
+\sum_{i=1}^n\frac{1}{\s_k^{(n+m)}}\left(\frac{\s_{k;i}^{(n+m)}}{\tl_i}\xi_i\right)^2.
\end{aligned}
\end{equation}
Then, from \eqref{sk-1 i} and \eqref{First Two term} it follows that
\[
\begin{aligned}
Q=&-\sum_{p\neq q} \frac{\s_{k;pq}^{(n+m)}}{\tl_p\tl_q}\xi_p \xi_q+\sum_{p, q}\frac{1}{\s_k^{(n+m)}}\frac{\s_{k;p}^{(n+m)}}{\tl_p}\frac{\s_{k;q}^{(n+m)}}{\tl_q}\xi_p\xi_q\\
&- \frac{1}{2n} \frac{\s_{k}^{(n+m)} \xi_1^2}{\tl_1^2} +\sum_{i=2}^n\frac{\s_k^{(n+m)}}{\tl_i^2}\xi_i^2 - \left(1-\frac{1}{2n}\right) \frac{\s_{k;1}^{(n+m)}}{\tl_1^2}\xi_1^2\\
&-2\sum_{i=2}^n\frac{\s_{k;i}^{(n+m)}}{\tl_i^2}\xi_i^2+2 \sum_{i > 1} \frac{\left(\s_k^{(n+m)}-\s_{k;i}^{(n+m)}\right)\xi_i^2}{\tl_i(\tl_1-\tl_i+\la_1^{-1})}.
\end{aligned}
\]
That is
\begin{equation}
\label{k,Q1}
\begin{aligned}
Q =&-\sum_{p\neq q} \frac{\s_{k;pq}^{(n+m)}}{\tl_p\tl_q}\xi_p \xi_q+\frac{1}{\s_k^{(n+m)}}\left(\sum_{p=1}^n\frac{\s_{k;p}^{(n+m)}}{\tl_p}\xi_p\right)^2\\
&- \frac{1}{2n} \frac{\s_{k}^{(n+m)} \xi_1^2}{\tl_1^2} +\s_k^{(n+m)}\sum_{i=2}^n\frac{\tl_1+\tl_i+\la_1^{-1}}{(\tl_1-\tl_i+\la_1^{-1})\tl_i^2}\xi_i^2\\
&- \left(1-\frac{1}{2n}\right)\frac{\s_{k;1}^{(n+m)}}{\tl_1^2}\xi_1^2-2\sum_{i=2}^n\frac{\s_{k;i}^{(n+m)} (\tl_1 + \la_1^{-1} )}{(\tl_1-\tl_i+\la_1^{-1}) \tl_i^2}\xi_i^2.
\end{aligned}
\end{equation}
Let $k=n-1$. Since $\tilde{y_i}=\frac{y_i}{\la_1}$ for $i=1,\dots,m$, we have 
\begin{align*}
   \s_{n-1;pq}^{(n+m)} (\hat \la) \sim O(\la_1^{-1})
\end{align*}
and 
\begin{align*}
\s_{n-1;p}^{(n+m)}(\hat{\la})=\s_{n-1;p}^{(n)} (\tl)+\sum_{i=1}^m \tilde a_i\s_{n-i-1;p}^{(n)}(\tl)\sim \frac{\s_n^{(n)}(\tl)}{\tilde{\la}_p}+O(\la_1^{-1}),
\end{align*}
where $\tilde a_i = e_i (\tilde y)$ for $i = 1,\ldots, m$.
Let 
\[
\eta_p=\frac{\xi_p}{\tl_p^2}, \;\mbox{where}\; 1 \leqslant p \leqslant n. 
\]
We obtain from \eqref{k,Q1} that
\begin{equation}\label{k,Q2}
\begin{aligned}
Q\gtrsim &\; \frac{1}{\s_{n-1}^{(n+m)}(\hat \la) } \left(\sum_{p=1}^n\left(\s_n^{(n)} +\sum_{i=1}^m \tilde a_i \tl_p\s_{n-i-1;p}^{(n)} \right)\eta_p\right)^2\\
&-\left(1-\frac{1}{2n}\right)\s_n^{(n)} \tl_1 \eta_1^2 - \sum_{i=2}^{n-1}\frac{ 2\s_n^{(n)} \tl_1\tl_i + O(\la_1^{-1})}{\tl_1-\tl_i+\la_1^{-1}}\eta_i^2\\
&-\frac{2\s_n^{(n)} \tl_1\tl_n}{\tl_1-\tl_n + \la_1^{-1}}  \eta_n^2 -O(\la_1^{-1})|\eta|^2.
\end{aligned}
\end{equation}
Let 
\[
e = (1,\dots,1)\in \mathbb{R}^n \; \mbox{and} \; \tau_p=\frac{1}{\tl_p}.
\]
We claim that, at $\tilde \la$, 
\begin{equation}\label{k,Q3}
\begin{aligned}
  \tilde{Q} :=&\; -\left(1-\frac{1}{2n}\right)\frac{\s_n^{(n)}}{\tau_1}\zeta_1^2-2\sum_{i=2}^{n-1}\frac{\s_n^{(n)}}{\tau_i-\tau_1+(\tl_i\la_1)^{-1}+\frac{1}{6n}}\zeta_i^2\\
  &\; -2\frac{\s_n^{(n)}}{\tau_n-\tau_1+(\tl_n\la_1)^{-1}}\zeta_n^2
\geqslant c|\zeta|^2
\end{aligned}
\end{equation}
for any vector 
\begin{align}\label{gam perp}
    \zeta\perp e,
\end{align}
 where $c>0$ depends only on $n$ and $\delta$. 
 Next we note that the positive coefficients of \eqref{k,Q3}, 
\begin{align*}
    \min\left\{ -\left(1-\frac{1}{2n}\right)\frac{\s_n^{(n)}}{\tau_1},  -2\frac{\s_n^{(n)}}{\tau_i-\tau_1+(\tl_i\la_1)^{-1}+\frac{1}{6n}}\right\} \geqslant -\frac{1}{c_0}\s_n^{(n)} >0
\end{align*}
  for any $i=2,\dots, n-1$, and the negative coefficients of \eqref{k,Q3}, 
  \begin{align*}
     0>-2\frac{\s_n^{(n)}}{\tau_n-\tau_1+(\tl_n\la_1)^{-1}} \geqslant   c_0\s_n^{(n)}, 
  \end{align*} 
  where $c_0>1$ only depends on $n$ and $\delta$. 
  If 
  \begin{align}\label{gam n}
      |\zeta|\geqslant 2nc_0|\zeta_n|,
  \end{align}
then there exists some $1\leqslant i\leqslant n-1$, such that $|\zeta_i|\geqslant \frac{1}{n}|\zeta|$ and 
\begin{align}\label{Q1}
\tilde Q\geqslant -\s_n^{(n)} \frac{|\zeta|^2}{n^2c_0} + \s_n^{(n)} \frac{|\zeta|^2}{4n^2c_0} \geqslant -\s_n^{(n)} \frac{|\zeta|^2}{2n^2c_0}.
\end{align}

Otherwise, we assume that 
\begin{align}\label{gam n reverse}
|\zeta|< 2nc_0|\zeta_n|.    
\end{align}
\eqref{gam perp} implies that
\begin{align}\label{gam perp2}
   \sum_{i=1}^{n-1}\zeta_i=-\zeta_n.
\end{align}
By Cauchy-Schwarz inequality,
\begin{align}\label{CS ineq}
    \sum_{i=1}^{n-1}\frac{\zeta_i^2}{b_i}\geqslant \frac{(\sum_{i=1}^{n-1}\zeta_i)^2}{\sum_{i=1}^{n-1}b_i}
\end{align}
for any $b_i>0$.
Denote by
\begin{align}\label{CS a}
  b_1=\frac{2n\tau_1}{2n-1},\quad b_i=\frac{\tau_i-\tau_1+(\tl_i\la_1)^{-1}+\frac{1}{6n}}{2}, \quad 2\leqslant i\leqslant n-1. 
\end{align}
And
\begin{equation}
\begin{aligned}\label{a}
  \sum_{i=2}^{n-1}b_i =&\; \frac{2n}{2n-1}\tau_1+\frac{\s_1(\tau)-(n-1)\tau_1-\tau_n}{2} +\sum_{i=1}^{n-1}\frac{(\tl_i\la_1)^{-1}}{2}+\frac{n-2}{12n} \\
    =&\; \frac{2n}{2n-1}\tau_1+\frac{\frac{\s_{n-1}}{\s_n}(\tl)-(n-1)\tau_1-\tau_n}{2} +\frac{n-2}{12n}+\sum_{i=1}^{n-1}\frac{(\tl_i\la_1)^{-1}}{2} \\
    \sim&\; -\frac{\tau_n}{2}+\left(\frac{1}{2n-1}-\frac{n-3}{2}\right)\tau_1 +\frac{1}{12}-\frac{1}{6n}+O(\la_1^{-1}) \\
     \leqslant &\; \frac{1}{2} \left(-\tau_n+\frac{2\tau_1}{3}-\frac{1}{3n}\right),
\end{aligned}    
\end{equation}
where we use the fact that $n\geqslant 3$.
Combining \eqref{k,Q3}, \eqref{CS ineq} and \eqref{a}, we have 
\begin{equation}\label{k,Q4}
\begin{aligned}
\tilde{Q} \geqslant&-2\s_n^{(n)} \left(\frac{1}{-\tau_n+\frac{2\tau_1}{3}-\frac{1}{3n}}+\frac{1}{\tau_n-\frac{2}{3}\tau_1}\right)\zeta_n^2\\
    =& -\frac{2}{3n}\s_n^{(n)} \zeta_n^2\frac{\tau_1}{(\tau_n-\frac{2\tau_1}{3}+\frac{1}{3n})(\tau_n- \frac{2}{3} \tau_1)}\geqslant 
    - c(n) \s_n^{(n)} \zeta_n^2,
\end{aligned}
\end{equation}
where we use \eqref{ln} in the last inequality. Then from \eqref{gam n reverse} and \eqref{k,Q4}, we have 
\begin{align}\label{Q2}
    \tilde{Q}\geqslant - \frac{c(n) \s_n^{(n)}}{4n^2c_0^2}|\zeta|^2.
\end{align}
Note that $-\s_n^{(n)} (\tl)=-\tl_1\cdots\tl_n\geqslant C_{\delta}$ by \eqref{ln}.
Combining \eqref{Q1} and \eqref{Q2}, we prove the claim \eqref{k,Q3}.

Now we are ready to explain that \eqref{k,Q2} is nonnegative.\\
If $\eta\perp e$, then due to \eqref{k,Q3},
\eqref{k,Q2} reduces to
\begin{equation*}
\begin{aligned}
Q\gtrsim &\; - \left(1-\frac{1}{2n}\right)\s_n^{(n)} \tl_1\eta_1^2 + 2\s_n^{(n)} \sum_{i=2}^{n-1}\frac{\tl_1\tl_i}{\tl_1-\tl_i+\la_1^{-1}+\frac{1}{6n}}\eta_i^2\\
&\; +
2\s_n^{(n)}\frac{\tl_1\tl_n}{\tl_1-\tl_n+\la_1^{-1}}\eta_n^2
 - O(\la_1^{-1})|\eta|^2
\geqslant 0,
\end{aligned}
\end{equation*}
where we define $\eta$ by $\eta_p=\frac{\xi_p}{\tl_p^2}.$
If $\eta \not\perp e$, then by appropriate normalization we only need to prove that $Q$ is nonnegative definite for any direction vector $\eta=\zeta + e$, where $\zeta\perp e$. This is because $Q$ is homogeneous of degree 2 (see equation \eqref{k}). Due to \eqref{k,Q3}, \eqref{k,Q2} reduces
\begin{equation}\label{k,Q6}
\begin{aligned}
Q\gtrsim &\; \frac{1}{\s_{n-1}^{(n+m)} (\hat \la)} \left(n\s_n^{(n)} +\sum_{p=1}^n\sum_{i=1}^m \tilde a_i\tl_p\s_{n-i-1;p}^{(n)} \eta_p\right)^2\\
&\;- \left\{\left(1-\frac{1}{2n}\right)\s_n^{(n)} \tl_1\eta_1^2 + 2\s_n^{(n)} \sum_{i=2}^{n-1}\frac{\tl_1\tl_i}{\tl_1-\tl_i+\la_1^{-1}+\frac{1}{6n}}\eta_i^2\right.\\
&\;\left. +
2\s_n^{(n)}\frac{\tl_1\tl_n}{\tl_1-\tl_n+\la_1^{-1}}\eta_n^2\right\}
 - O(\la_1^{-1})(|\zeta|^2+n) \\
\geqslant &\; \frac{1}{\s_{n-1}^{(n+m)}(\hat \la)} \left(n\s_n^{(n)} +\sum_{p=1}^n\sum_{i=1}^m \tilde a_i\tl_p\s_{n-i-1;p}^{(n)} \eta_p\right)^2
+ \frac{c|\zeta|^2}{2}-C,
\end{aligned}
\end{equation}
where we assume that $\la_1$ is sufficiently large depending only on $n$ and $(a_1,\dots,a_m)$. If $|\zeta|^2\geqslant\frac{2C}{c}$, then \eqref{k,Q6} is positive since the first term is a perfect square.
Thus, we assume that $|\zeta|^2<\frac{2C}{c}$. Then due to \eqref{k,Q6} we have 
\begin{align*}
Q\geqslant \frac{1}{\s_{n-1}^{(n+m)} (\hat \la)}\left(n\s_n^{(n)} +\sum_{p=1}^n\sum_{i=1}^m \tilde a_i\tl_p\s_{n-i-1;p}^{(n)} \eta_p\right)^2-C > 0.   
\end{align*}
Here we use the fact that 
\[
\tilde a_i\tl_p\s_{n-i-1;p}^{(n)} (\tilde \la)\eta_p\sim O(\la_1^{-1})\;
\mbox{and}\;
\s_{n-1}^{(n+m)}(\hat \la)\sim O(\la_1^{-(n-1)}).
\]
The proof is completed.
\end{proof}

\section{Condition \eqref{con2} in Theorems \ref{thm1} and \ref{thm2}}
In fact, we can also derive the same Pogorelov interior estimates for Theorem \ref{thm1} and Theorem \ref{thm2} under Condition \eqref{con2}: $a_i\geqslant 0$ and $\la\in \Gamma_{k-1}^n$ by noting that 
\begin{align}\label{kF2}
\sum_{i=1}^n F^{ii}u_{ii}=k \sigma_k^{(n)}+\sum_{i=1}^m(k-i)a_i\sigma_{k-i}^{(n)} \leqslant kF.
\end{align}
In addition, we can estimate $\sum_{i=1}^n F^{ii}$ as 
\begin{align*}
    \sum_{i=1}^n F^{ii}=(n-k+1)\s_{k-1}^{(n)} + \sum_{i=1}^m (n-k+i+1)a_i\s_{k-i-1}^{(n)}.
\end{align*}
Using \eqref{new k-1}, we have
\begin{align*}
    \s_{k-1}^{(n+m)}=\s_{k-1}^{(n)}+\sum_{i=1}^ma_i \s_{k-i-1}^{(n)} \geqslant C(\s_1^{(n+m)})^{\frac{1}{k-1}}\geqslant C\la_1^{\frac{1}{k-1}}.
\end{align*}
There exists some $0\leqslant i\leqslant m$, such that $a_i \s_{k-i-1}^{(n)} \geqslant \frac{C}{m+1}\la_1^{\frac{1}{k-1}}$ with $a_0=1$. Then we obtain 
\begin{align}
\label{sumFi2}
    \sum_{i=1}^n F^{ii}\geqslant C \la_1^{\frac{1}{k-1}}.
\end{align}
By replacing \eqref{kF} with \eqref{kF2} and \eqref{sumFi} with \eqref{sumFi2},
 we complete the proof of Theorem \ref{thm1} and Theorem \ref{thm2} under Condition \eqref{con2} by the same argument.


\begin{thebibliography}{99}

\bibitem{And07} B. Andrews, Pinching estimates and motion of hypersurfaces by curvature functions, J. Reine Angew. Math. 608 (2007), 17--33.
	
	
	
\bibitem{BCGJ}  J. Bao, J. Chen, B. Guan and M. Ji, Liouville property and regularity of a Hessian quotient equation, Amer. J. Math. 125 (2003), no. 2, 301–316.	

\bibitem{Ca} E. Calabi, 
Improper affine hyperspheres of convex type and a generalization of a theorem by K. J\"orgens,
Michigan Math. J. 5 (1958), 105-126.

\bibitem{CY}  S.-Y. A. Chang and Y. Yuan, A Liouville problem for the sigma-2 equation, Discrete Contin. Dyn. Syst. 28 (2010), no. 2, 659-664.

\bibitem{CDH} C.Q. Chen, W.S. Dong and F. Han, Interior Hessian estimates for a class of Hessian type equations, Calc. Var. Partial Differential Equations 62 (2023), no. 2, Paper No. 52, 15 pp.

\bibitem{CHO} C.Q. Chen, F. Han and Q.Z. Ou, The interior $C^2$ estimate for the Monge-Amp\`ere equation in dimension $n=2$, Anal. PDE 9 (2016), no. 6, 1419-1432.

\bibitem{CX} L. Chen and N. Xiang, Rigidity theorems for the entire solutions of 2-Hessian equation, J. Differential Equations 267 (2019), no. 9, 5202–5219.

\bibitem{CYau} S.-Y. Cheng and S.-T. Yau, Complete affine hypersurfaces. I. The completeness of affine metrics, Comm. Pure Appl. Math. 39 (1986), no. 6, 839-866.


\bibitem{CW} K.-S. Chou and X.-J. Wang, A variational theory of the Hessian equation. Comm. Pure Appl. Math. 54 (2001), no. 9, 1029–1064.


\bibitem{CD} J.C. Chu and S. Dinew, Liouville theorem for a class of Hessian equations, arXiv:2306.13825.

\bibitem{CJ} J.C. Chu and H. Jiao, Curvature estimates for a class of Hessian type equations, Calc. Var. Partial Differential Equations 60 (2021), Paper No. 90, 18 pp.

\bibitem{Dinew} S. Dinew, Interior estimates for p-plurisubharmonic functions, Indiana Univ. Math. J. 72 (2023), no. 5, 2025-2057.


\bibitem{D} H.J. Dong, Hessian equations with elementary symmetric functions, Comm. Partial Differential Equations 31 (2006), no. 7-9, 1005--1025.

\bibitem{Dong} W.S. Dong, Curvature estimates for $p$-convex hypersurfaces of prescribed curvature, Rev. Mat. Iberoam. 39 (2023), no. 3, 1039-1058.



\bibitem{Du} S.-Z. Du, Necessary and sufficient conditions to Bernstein theorem of a Hessian equation, Trans. Amer. Math. Soc. 375 (2022), no. 7, 4873–4892.





\bibitem{GQ} P.F. Guan and G.H. Qiu, Interior $C^2$ regularity of convex solutions to prescribing scalar curvature equations, Duke Math. J. 168 (2019), no. 9, 1641–1663. 

\bibitem{GRW} P.F. Guan, C.Y. Ren and Z.Z. Wang, Global $C^2$ estimates for convex solutions of curvature equations, Comm. Pure Appl. Math. 68 (2015), no. 8, 1287-1325.


\bibitem{GZ} P.F. Guan and X.W. Zhang, A class of curvature type equations, Pure Appl. Math. Q. 17 (2021), no. 3, 865-907. 

\bibitem{HL} F.R. Harvey and H.B. Lawson Jr., Calibrated geometries, Acta Math. 148 (1982), 47--157.


\bibitem{Heinz} E. Heinz,
On elliptic Monge-Amp\`ere equations and Weyl's embedding problem, J. Analyse Math. 7 (1959), 1-52.

\bibitem{HZ} H. Hong and R.J. Zhang, Curvature estimates for semi-convex solutions of the asymptotic Plateau problem in $\mathbb{H}^{n+1}$, arxiv.org/abs/2408.09428.


\bibitem{HMW} Z.L. Hou, X.-N. Ma and D.M. Wu, A second order estimate for complex Hessian equations on a compact K\"ahler manifold, Math. Res. Lett. 17 (2010), 547-561. 

\bibitem{HS} G. Huisken and C. Sinestrari, Convexity estimates for mean curvature flow and singularities of mean convex surfaces, Acta. Math. 183 (1999), 45-70.


\bibitem{JT} F. Jiang and N. Trudinger, On Pogorelov estimates in optimal transportation and geometric optics, Bull. Math. Sci. 4 (2014), no. 3, 407-431.


\bibitem{JX} C. Jin and L. Xu, The Pogorelov type $C^2$ estimates for sum Hessian equations,
J. Differential Equations 453 (2026), part 2, Paper No. 113853, 23 pp.

\bibitem{Jo} K. J\"orgens,
\"Uber die L\"osungen der Differentialgleichung $rt-s^2=1$,
Math. Ann. 127 (1954), 130-134.


\bibitem{K} N. Korevaar, A priori interior gradient bounds for solutions to elliptic Weingarten equations, Ann. Inst. H. Poincar\'e Anal. Non Lin\'eaire 4 (1987), no. 5, 405–421.


\bibitem{Kr} N.V. Krylov, On the general notion of fully nonlinear second-order elliptic equations, Trans. Amer. Math. Soc. {\bf 347} (1995), no.~3, 857--895.

\bibitem{LRW2} C.H. Li, C.Y. Ren and Z.Z. Wang, Curvature estimates for convex solutions of some fully nonlinear Hessian-type equations, Calc. Var. Partial Differential Equations 58 (2019), no. 6, Paper No. 188, 32 pp. 
\bibitem{LRW} M. Li, C.Y. Ren and Z.Z. Wang,
An interior estimate for convex solutions and a rigidity theorem, J. Funct. Anal. 270 (2016), 2691-2714.

\bibitem{LR} P.F. Li, C.Y. Ren, Pogorelov type $C^2$ estimates for sum Hessian equations, arXiv:2504.06711.


\bibitem{YYLi} Y.Y. Li, Some existence results of fully nonlinear elliptic equations of Monge-Amp\`ere type, Comm. Pure Appl. Math. 43 (1990), 233-271.


\bibitem{LYZ} W. Liang, J. Yan and H. Zhu, The Pogorelov estimates for the sum Hessian equation with rigidity theorem and parabolic versions, arXiv:2412.11822.

\bibitem{Liu} J.K. Liu, Interior $C^2$ estimate for Monge-Amp\`ere equation in dimension two, Proc. Amer. Math. Soc. 149 (2021), no. 6, 2479-2486.


\bibitem{LiuT} J.K. Liu and N. Trudinger, On Pogorelov estimates for Monge-Amp\`ere type equations, Discrete Contin. Dyn. Syst. 28 (2010), no. 3, 1121-1135.



\bibitem{LR2} Y. Liu and C.Y. Ren, Pogorelov type $C^2$ estimates for sum Hessian equations and a rigidity theorem, J. Funct. Anal.  284 (2023), no. 1, Paper No. 109726, 32 pp.

\bibitem{LTr}
 M. Lin and N. Trudinger, On some inequalities for elementary symmetric functions, Bull. Austral. Math. Soc. 50:2 (1994), 317–326.


\bibitem{LT} S.Y. Lu and Y.-L. Tsai, A simple proof of curvature estimates for the $n-1$ Hessian equation,  Proc. Amer. Math. Soc. 154 (2026), no. 2, 893–904.


\bibitem{LT2} S.Y. Lu and Y.-L. Tsai, Pogorelov type interior $C^2$ estimate for Hessian quotient equation and its application, J. Reine Angew. Math. 831 (2026), 155-184.

\bibitem{Po1} A.V. Pogorelov, 
The regularity of the generalized solutions of the equation $\det(\partial^2 u/\partial x^i \partial x^j)=\varphi(x^1, x^2,\ldots, x^n)>0$,
Dokl. Akad. Nauk SSSR 200 (1971), 534-537.
 
\bibitem{Po2} A.V. Pogorelov, On the improper convex affine hyperspheres,
Geometriae Dedicata 1 (1972), no. 1, 33-46.

\bibitem{Po3} A.V. Pogorelov, The Minkowski multidimensional problem, Scripta Ser. Math., V. H. Winston \& Sons, Washington 1978. 

\bibitem{RW1} C.Y. Ren and Z.Z. Wang, On the curvature estimates for Hessian equations, Amer. J. Math. 141 (2019), 1281-1315.

\bibitem{RW2} C.Y. Ren and Z.Z. Wang, The global curvature estimate for the $n-2$ Hessian equation, Calc. Var. Partial Differential Equations 62 (2023), no. 9, Paper No. 239, 50 pp.


\bibitem{SY} R. Shankar and Y. Yuan, Rigidity for general semiconvex entire solutions to the sigma-2 equation, Duke Math. J. 171 (2022), no. 15, 3201–3214.

\bibitem{SY2} R. Shankar and Y. Yuan, Hessian estimates for the sigma-2 equation in dimension four, Ann. of Math. (2), 201 (2025), no. 2, 489-513.

\bibitem{SUW} W. Sheng, J. Urbas and X.-J. Wang, Interior curvature bounds for a class of curvature equations, Duke Math. J. 123 (2004), no. 2, 235-264.


\bibitem{TruW}  N. Trudinger and X.-J. Wang, The Monge-Amp\`ere equation and its geometric applications. Handbook of geometric analysis. No. 1, 467-524, Adv. Lect. Math. (ALM), 7, Int. Press, Somerville, MA, 2008. 

\bibitem{Tu} Q. Tu, Pogorelov type estimates for $(n-1)$-Hessian equations and related rigidity theorems, arXiv:2405.02939.


\bibitem{Warren} M. Warren, Nonpolynomial entire solutions to $\sigma_k$ equations, Comm. Partial
Differential Equations 41 (2016), no. 5, 848–853.


\bibitem{Yuan} Y. Yuan, A monotonicity approach to Pogorelov's Hessian estimates for Monge-Amp\`ere equation, Math. Eng. 5 (2023), no. 2, Paper No. 037.


\bibitem{Zhang} R.J. Zhang, $C^2$ estimates for $k$-Hessian equations and a rigidity theorem, Adv. Math. 480 (2025), part A, Paper No. 110488.



\end{thebibliography}
\end{document}